\documentclass{amsart} 
\usepackage{a4wide}
\usepackage{setspace}
\usepackage{amsmath,amsthm,amssymb}
\usepackage{color}
\usepackage{bbm}
\usepackage[dvips]{graphicx}
\usepackage{longtable}
\usepackage{url}
\usepackage{enumerate}
\usepackage{float}

\providecommand{\MR}{\relax\ifhmode\unskip\space\fi MR }

\newcommand{\transp}{{\scriptscriptstyle \top}}
\newtheorem{thm}{\bf{Theorem}}[section]
\newtheorem{lemma}[thm]{\bf{Lemma}}

\newtheorem{rem}[thm]{\bf{Remark}}

\begin{document}

\title[Convergence analysis of DDP with cut selection]{Dual Dynamic Programming with cut selection: convergence proof and numerical experiments}

\author{Vincent Guigues}

\maketitle

\begin{center}
    Funda\c{c}\~ao Getulio Vargas, School of applied mathematics\\
            190 Praia de Botafogo, Rio de Janeiro, Brazil,\\
            Tel: 55 21 3799 6093, Fax: 55 21 3799 5996, {{\tt vguigues@fgv.br}}
\end{center}


\begin{abstract} We consider convex optimization problems formulated using dynamic programming
equations. Such problems can be solved using the 
Dual Dynamic Programming algorithm combined with the Level 1 cut selection strategy or the Territory algorithm
to select the most relevant Benders cuts.
We propose a limited memory variant of Level 1 and show the convergence
of DDP combined with the Territory algorithm, Level 1 or its variant for nonlinear
optimization problems. In the special case of linear programs, we show convergence in a finite
number of iterations.
Numerical simulations illustrate the interest of our variant
and show that it can be much quicker than a simplex algorithm on some large instances
of portfolio selection and inventory problems.
\end{abstract}


\section{Introduction}

Dual Dynamic Programming (DDP) is a decomposition algorithm to solve
some convex optimization problems. 
The algorithm computes lower approximations of the cost-to-go functions expressed as a supremum of affine functions called
optimality cuts. Typically, at each iteration, a fixed number of cuts is added for each cost-to-go function.
It is the deterministic counterpart of the Stochastic Dual Dynamic Programming (SDDP) algorithm pioneered by \cite{pereira}.
SDDP is still studied and has been the object of several
recent improvements and extensions \cite{shapsddp}, \cite{philpmatos}, \cite{guiguesrom10}, \cite{guiguesrom12}, \cite{guiguescoap2013},
\cite{kozmikmorton}, \cite{shaptekaya2}, \cite{dpcuts0}, \cite{pfeifferetalcuts}. 
In particular, these last three references discuss strategies for selecting the most relevant optimality cuts
which can be applied to DDP.
In stochastic optimization, the problem of cut selection for lower approximations of the cost-to-go functions
associated to each node of the scenario tree was discussed for the first time in \cite{rusparallel93} where only
the active cuts are selected.
Pruning strategies of basis (quadratic) functions have been proposed in \cite{gaubenezheng} and \cite{mcdesgaube}
for max-plus based approximation methods which, similarly to SDDP, approximate the cost-to-go functions of a nonlinear
optimal control problem by a supremum of basis functions. More precisely, in \cite{gaubenezheng},
a fixed number of cuts is pruned and cut selection is done solving a combinatorial optimization problem.
For SDDP, in \cite{shaptekaya2} it is suggested at some iterations to eliminate redundant cuts
(a cut is redundant if it is never active in describing the lower approximate cost-to-go function).
This procedure is called {\em{test of usefulness}} in \cite{pfeifferetalcuts}.
This requires solving at each stage as many linear programs as there are cuts. 
In \cite{pfeifferetalcuts} and \cite{dpcuts0}, only the cuts that have the largest value for at least
one of the trial points computed are considered relevant, see Section \ref{algorithmddp} for details. 
This 
strategy is called the {\em{Territory algorithm}} in \cite{pfeifferetalcuts}
and Level 1 cut selection in \cite{dpcuts0}. It was presented for the first time
in 2007 at the ROADEF congress by David Game and Guillaume Le Roy (GDF-Suez), see \cite{pfeifferetalcuts}.
However, a difference between  \cite{pfeifferetalcuts} and \cite{dpcuts0}
is that in \cite{pfeifferetalcuts} the nonrelevant cuts are pruned whereas 
in \cite{dpcuts0} all computed cuts are stored and the relevant cuts are selected from this set of cuts. 

In this context the contributions of this paper are as follows.
We propose a limited memory variant of Level 1. 
We study the convergence of DDP combined with a class of cut selection strategies that satisfy an assumption
(Assumption (H2), see Section \ref{csmethods}) satisfied by the Territory algorithm, the Level 1 cut selection strategy, and 
its variant. In particular, the analysis applies to (i) mixed cut selection strategies 
that use the Territory
algorithm, Level 1, or its variant to select a first set of cuts and then apply the test of usefulness to these cuts and
to (ii) the Level $H$
cut selection strategy from \cite{dpcuts0} that keeps at each trial point the $H$ cuts having the largest values.
In the case when the problem is linear, we additionally show convergence in a finite number of iterations.
Numerical simulations show the interest of the proposed limited memory variant of Level 1  
and show that it can be much more efficient than a simplex algorithm on some instances
of portfolio selection and inventory problems.

The outline of the study is as follows. 
Section \ref{propvaluefunction} recalls from \cite{guigues2014cvsddp} a formula
for the subdifferential of the value function of a convex optimization problem. It is useful for the implementation and convergence analysis of
DDP applied to convex nonlinear problems.
Section \ref{pbform} describes the class of problems considered and assumptions.
Section \ref{ddpwithout} recalls the DDP algorithm while Section \ref{csmethods} recalls
Level 1 cut selection, the Territory algorithm and describes the limited memory variant we propose.
Section \ref{convanalysis} studies the convergence of DDP with cut selection applied to nonlinear problems
while Section \ref{convanalysislp} studies the convergence for linear programs.
Numerical simulations are reported in Section  \ref{numsim}.

We use the following notation and terminology:
\par - The usual scalar product in $\mathbb{R}^n$ is denoted by $\langle x, y\rangle = x^\transp y$ for $x, y \in \mathbb{R}^n$.
\par - $\mbox{ri}(A)$ is the relative interior of set $A$.
\par - $\mathbb{B}_n$ is the closed unit ball in $\mathbb{R}^n$.
\par - dom($f$) is the domain of function $f$.
\par - $|I|$ is the cardinality of the set $I$.
\par - $\mathbb{N}^*$ is the set of positive integers.

\section{Formula for the subdifferential of the value function of a convex optimization problem} \label{propvaluefunction}

We recall from \cite{guigues2014cvsddp} a formula for the subdifferential
of the value function of a convex optimization problem. It plays a central role in
the implementation and convergence analysis of DDP method applied to convex problems
and will be used in the sequel.

Let $\mathcal{Q}: X\rightarrow {\overline{\mathbb{R}}}$, be the value function given by
\begin{equation} \label{vfunctionq}
\mathcal{Q}(x)=\left\{
\begin{array}{l}
\inf_{y \in \mathbb{R}^{n}} \;f(x,y)\\
y \in S(x):=\{y\in Y \;:\;Ax+By=b,\;g(x,y)\leq 0\}.
\end{array}
\right.
\end{equation}
Here, $A$ and $B$ are matrices of appropriate dimensions, and
$X \subseteq \mathbb{R}^m$ and $Y \subseteq \mathbb{R}^n$ are nonempty, compact, and convex
sets. Denoting by
\begin{equation}\label{epsfatten}
X^\varepsilon := X + \varepsilon  \mathbb{B}_m
\end{equation}
the $\varepsilon$-fattening of the set $X$, we make the following assumption (H):
\begin{itemize}
\item[1)] $f:\mathbb{R}^m \small{\times} \mathbb{R}^n \rightarrow \mathbb{R} \cup \{+\infty\}$ is 
lower semicontinuous, proper, and convex.
\item[2)] For $i=1,\ldots,p$, the $i$-th component of function
$g(x,y)$ is a convex lower semicontinuous function
$g_i:\mathbb{R}^m \small{\times} \mathbb{R}^n \rightarrow \mathbb{R} \cup \{+\infty\}$.
\item[3)] There exists $\varepsilon>0$ such that $X^{\varepsilon}{\small{\times}}  Y \subset \mbox{dom}(f)$.
\end{itemize}

Consider the dual problem 
\begin{equation}\label{dualpb}
\displaystyle \sup_{(\lambda, \mu) \in \mathbb{R}^q \small{\times} \mathbb{R}_{+}^{p} }\; \theta_{x}(\lambda, \mu)
\end{equation}
for the dual function
$$
\theta_{x}(\lambda, \mu)=\displaystyle \inf_{y \in Y} \;f(x,y) + \lambda^\transp (Ax+By-b) + \mu^\transp g(x,y).
$$
We denote by $\Lambda(x)$ the set of optimal solutions of the  dual problem \eqref{dualpb}
and we use the notation
$$
\mbox{Sol}(x):=\{y \in S(x) : f(x,y)=\mathcal{Q}(x)\}
$$
to indicate the solution set to \eqref{vfunctionq}.

It is well known that under Assumption (H), $\mathcal{Q}$ is convex.
The description of the subdifferential of $\mathcal{Q}$ is given in
the following lemma:
\begin{lemma}[Lemma 2.1 in \cite{guigues2014cvsddp}]\label{dervaluefunction} 
Consider the value function $\mathcal{Q}$ given by \eqref{vfunctionq} and take $x_0 \in X$
such that $S(x_0)\neq \emptyset$.
Let Assumption (H) hold and
assume the Slater-type constraint qualification condition:
$$
there \;exists\; (\bar x, \bar y) \in X{\small{\times}}\emph{ri}(Y) \mbox{ such that }A \bar x + B \bar y = b \mbox{ and } (\bar x, \bar y) \in \emph{ri}(\{g \leq 0\}).
$$
Then $s \in \partial \mathcal{Q}(x_0)$ if and only if
\begin{equation}\label{caractsubQ}
\begin{array}{l}
(s, 0) \in  \partial f(x_0, y_0)+\Big\{[A^\transp; B^\transp ] \lambda \;:\;\lambda \in \mathbb{R}^q\Big\}\\
\hspace*{1.2cm}+ \Big\{\displaystyle \sum_{i \in I(x_0, y_0)}\; \mu_i \partial g_i(x_0, y_0)\;:\;\mu_i \geq 0 \Big\}+\{0\}\small{\times}\mathcal{N}_{Y}(y_0),
\end{array}
\end{equation}
where $y_0$ is any element in the solution set \mbox{Sol}($x_0$) and with
$$I(x_0, y_0)=\Big\{i \in \{1,\ldots,p\} \;:\;g_i(x_0, y_0) =0\Big\}.$$

In particular, if $f$  and $g$ are differentiable, then 
$$
\partial \mathcal{Q}(x_0)=\Big\{  \nabla_x f(x_0, y_0)+ A^\transp \lambda + \sum_{i \in I(x_0, y_0)}\; \mu_i \nabla_x g_i(x_0, y_0)\;:\; (\lambda, \mu) \in \Lambda(x_0) \Big\}.
$$
\end{lemma}
\begin{proof}
See \cite{guigues2014cvsddp}. \hfill
\end{proof}

\section{Problem formulation} \label{pbform}

Consider the convex optimization problem
\begin{equation}\label{optconvexd}
\left\{
\begin{array}{l}
\displaystyle \inf_{x_1,\ldots,x_T}\;\sum_{t=1}^{T} f_t(x_{t}, x_{t-1}) \\
x_t \in \mathcal{X}_t,\;g_t(x_t, x_{t-1}) \leq 0,\;\;\displaystyle A_{t} x_{t} + B_{t} x_{t-1} = b_t, t=1,\ldots,T,
\end{array}
\right.
\end{equation}
for $x_0$ given and the corresponding dynamic programming equations
\begin{equation}\label{dpconvexd}
\mathcal{Q}_{t}(x_{t-1})=
\left\{
\begin{array}{l}
\displaystyle \inf_{x_t}\;F_t(x_{t}, x_{t-1}):=f_t(x_{t}, x_{t-1}) + \mathcal{Q}_{t+1}(x_{t})\\
x_t \in \mathcal{X}_t,\;g_t(x_t, x_{t-1}) \leq 0,\;\;\displaystyle A_{t} x_{t} + B_{t} x_{t-1} = b_t,
\end{array}
\right.
\end{equation}
for $t=1,\ldots,T$, with $\mathcal{Q}_{T+1} \equiv 0$, and $g_t(x_t, x_{t-1})=(g_{t, 1}(x_t, x_{t-1}),\ldots,g_{t, p}(x_t, x_{t-1}))$ with
$g_{t, i}: \mathbb{R}^n \small{\times} \mathbb{R}^n \rightarrow \mathbb{R} \cup \{+\infty\}$.

Cost-to-go function $\mathcal{Q}_{t+1}(x_{t})$ represents the optimal (minimal)
total cost for time steps $t+1, \ldots,T$, starting from state $x_t$ at the beginning of step $t+1$.

We make the following assumptions (H1):\\

\par (H1) Setting $\mathcal{X}_{t}^\varepsilon := \mathcal{X}_{t} + \varepsilon  \mathbb{B}_n$, for $t=1,\ldots,T,$
\begin{itemize}
\item[(a)] $\mathcal{X}_{t} \subset \mathbb{R}^n$ is nonempty, convex, and compact.
\item[(b)] $f_t$ is proper, convex, and lower semicontinuous.
\item[(c)] setting $g_t(x_t, x_{t-1})=(g_{t, 1}(x_t, x_{t-1}), \ldots, g_{t, p}(x_t, x_{t-1}))$, for 
$i=1,\ldots,p$, the $i$-th component function
$g_{t, i}(x_t , x_{t-1})$ is a convex lower semicontinuous function.
\item[(d)] There exists $\varepsilon>0$ such that
$
\mathcal{X}_{t} {\small{\times}}  \mathcal{X}_{t-1}^{\varepsilon} \subset \mbox{dom}(f_t)
$
and for every $x_{t-1} \in \mathcal{X}_{t-1}^{\varepsilon}$, there exists 
$x_t \in \mathcal{X}_t$ such that $g_t(x_t, x_{t-1}) \leq 0$ and $\displaystyle A_{t} x_{t} + B_{t} x_{t-1} = b_t$.
\item[(e)] If $t \geq 2$, there exists
$$
{\bar x}_{t}=({\bar x}_{t, t}, {\bar x}_{t, t-1}) \in 
\mbox{ri}(\mathcal{X}_{t}) \small{\times} \mathcal{X}_{t-1}
\cap \mbox{ri}(\{g_t \leq 0\})$$ such that $\bar x_{t, t} \in \mathcal{X}_t$,
$g_t(\bar x_{t, t}, \bar x_{t, t-1}) \leq 0$ and $A_{t} \bar x_{t, t} + B_{t} \bar x_{t, t-1} = b_t$.\\
\end{itemize}

\par {\textbf{Comments on the assumptions.}} Assumptions (H1)-(a), (H1)-(b), and (H1)-(c)
ensure that the cost-to-go functions $\mathcal{Q}_t, t=2,\ldots,T$ are convex.

Assumption (H1)-(d) guarantees that $\mathcal{Q}_t$ is finite on $\mathcal{X}_{t-1}^{\varepsilon}$
and has bounded subgradients on $\mathcal{X}_{t-1}$.
It also ensures that the cut coefficients are finite and therefore that the
lower piecewise affine approximations computed for $\mathcal{Q}_t$
by the DDP algorithm are convex and Lipschitz continuous on $\mathcal{X}_{t-1}^{\varepsilon}$
(see Lemmas \ref{contqt} and \ref{contqtk} below and \cite{guigues2014cvsddp} for a proof).

Assumption (H1)-(e) is used to obtain an explicit description of the subdifferentials of
the value functions (see Lemma \ref{dervaluefunction}). This description is necessary to obtain the coefficients of the cuts.

\section{Algorithm} \label{algorithmddp}

\subsection{DDP without cut selection}\label{ddpwithout}

The Dual Dynamic Programming (DDP) algorithm to be presented in this section 
is a decomposition method for solving problems of form \eqref{optconvexd}.
It relies on the convexity of recourse functions $\mathcal{Q}_t, t=1,\ldots,T+1$:
\begin{lemma}\label{contqt} 
Consider the optimization problem \eqref{optconvexd}
and recourse functions $\mathcal{Q}_t, t=1,\ldots,T+1$, given by 
\eqref{dpconvexd}. 
Let Assumptions (H1)-(a), (H1)-(b), (H1)-(c), and (H1)-(d) hold. Then
for $t=1,\ldots,T+1$, $\mathcal{Q}_t$ is convex,
finite on $\mathcal{X}_{t-1}^{\varepsilon}$,
and continuous on $\mathcal{X}_{t-1}$.
\end{lemma}
\begin{proof} 
See the proof of Proposition 3.1 in \cite{guigues2014cvsddp}.\hfill
\end{proof}

The DDP algorithm builds for each $t=2,\ldots,T+1$, a polyhedral lower approximation $\mathcal{Q}_t^k$ at iteration $k$
for $\mathcal{Q}_{t}$. We start with $\mathcal{Q}_{t}^0=-\infty$. At the beginning of iteration $k$, are available the lower polyhedral approximations
$$  
\mathcal{Q}_t^{k-1}(x_{t-1}) =\displaystyle \max_{j \in \mathcal{S}_t^{k-1}} \mathcal{C}_t^j \Big( x_{t-1}\Big) :=\theta_t^j + \langle \beta_t^{j}, x_{t-1} -  x_{t-1}^j\rangle,
$$
for $\mathcal{Q}_t$, $t=2,\ldots,T+1$, where $\mathcal{S}_t^{k-1}$ is a subset of $\{0,1,\ldots,k-1\}$, initialized
taking $\mathcal{S}_t^{0}=\{0\}$ and $\mathcal{C}_t^j$  is the cut computed at iteration $j$, as explained below.

Let ${\underline{\mathcal{Q}}}_{t}^k: \mathcal{X}_{t-1} \rightarrow \mathbb{R}$ be the function given by
$$
{\underline{\mathcal{Q}}}_{t}^k(x_{t-1} )=
\left\{
\begin{array}{l}
\displaystyle \inf_{x} \;F_t^{k-1}(x, x_{t-1}):=f_t(x, x_{t-1}) + \mathcal{Q}_{t+1}^{k-1}(x)\\
x \in \mathcal{X}_t,\;g_t(x, x_{t-1}) \leq 0,\;\;\displaystyle A_{t} x + B_{t} x_{t-1} = b_t.
\end{array}
\right.
$$
At iteration $k$, in a forward pass, for $t=1,\ldots,T$, we compute an optimal solution $x_t^k$ of
\begin{equation} \label{forward1}
{\underline{\mathcal{Q}}}_{t}^k(x_{t-1}^k )=
\left\{
\begin{array}{l}
\displaystyle \inf_{x} \;F_t^{k-1}(x, x_{t-1}^k):=f_t(x, x_{t-1}^k) + \mathcal{Q}_{t+1}^{k-1}(x)\\
x \in \mathcal{X}_t,\;g_t(x, x_{t-1}^k) \leq 0,\;\;\displaystyle A_{t} x + B_{t} x_{t-1}^k = b_t,
\end{array}
\right.
\end{equation}
starting from $x_0^{k}=x_0$ and knowing that $\mathcal{Q}_{T+1}^k=\mathcal{Q}_{T+1}\equiv 0$.
We have that $\mathcal{Q}_{t}^{k-1} \leq \mathcal{Q}_{t}$ for all $t$ and cut $\mathcal{C}_t^k$ is built for
$\mathcal{Q}_{t}$ in such a way that $\mathcal{Q}_{t}^{k} \leq \mathcal{Q}_{t}$ holds. For step $t$, since
$\mathcal{Q}_{t+1}^{k-1} \leq \mathcal{Q}_{t+1}$, we have $\mathcal{Q}_{t} \geq {\underline{\mathcal{Q}}}_{t}^{k}$. 
Setting $\theta_t^k= {\underline{\mathcal{Q}}}_{t}^k(x_{t-1}^k )$ and taking $\beta_t^k \in \partial {\underline{\mathcal{Q}}}_{t}^k(x_{t-1}^k)$,
it follows that
$$
\begin{array}{lll}
\mathcal{Q}_t(x_{t-1}) &\geq &{\underline{\mathcal{Q}}}_{t}^k(x_{t-1}) \geq  \mathcal{C}_t^k( x_{t-1} ) = \theta_t^k + \langle {\beta}_t^{k}, x_{t-1} -  x_{t-1}^k \rangle.
\end{array}
$$
Vector $\beta_t^k$ from the set $\partial {\underline{\mathcal{Q}}}_{t}^k(x_{t-1}^k)$ is computed using  Lemma \ref{dervaluefunction}. In the original description of DDP, the cut $\mathcal{C}_t^k$
is automatically added to the set of cuts that make up the approximation $\mathcal{Q}_t^k$, i.e., $\mathcal{S}_t^k=\{0,1,\ldots,k\}$.

\subsection{Cut selection methods}\label{csmethods}

\subsubsection{Level 1 and Territory algorithm}

We now describe the Territory algorithm, the Level 1 cut selection and its limited memory variant which
satisfy the following assumption:
\begin{itemize}
\item[(H2)] The height of the cutting plane approximations $\mathcal{Q}_t^{k}$ at the trial points $x_{t-1}^k$ is nondecreasing, i.e.,
for all $t=2,\ldots,T$, for all $k_1 \in \mathbb{N}^{*}$, for all $k_2 \geq k_1$, we have 
$\mathcal{Q}_t^{k_1} ( x_{t-1}^{k_1}  )\geq \mathcal{C}_t^{k_1} ( x_{t-1}^{k_1}  )$
and $\mathcal{Q}_t^{k_2} ( x_{t-1}^{j}  ) \geq \mathcal{Q}_t^{k_1} ( x_{t-1}^{j}  )$
for every $j=1,\ldots,k_2$.
\end{itemize}

To describe the Level 1 cut selection strategy, we introduce the set $I_{t, i}^k$ of cuts computed at iteration $k$ or before that have the largest value at 
$x_{t-1}^i$:
\begin{equation} \label{indexlevel1}
I_{t, i}^k = \operatorname*{arg\,max}_{\ell=1,\ldots,k} \mathcal{C}_t^{\ell}( x_{t-1}^i ).
\end{equation}
With this notation, with Level 1 cut selection strategy, the cuts that make up $\mathcal{Q}_t^k$ are the cuts
that have the largest value for at least one trial point, i.e., 
\begin{equation} \label{qtk}
\displaystyle \mathcal{Q}_t^k( x_{t-1} ) =  \max_{\ell \in \mathcal{S}_t^k} \, \mathcal{C}_t^\ell( x_{t-1}) \mbox { with }
\mathcal{S}_t^k = \displaystyle \bigcup_{i=1}^k I_{t, i}^{k}.
\end{equation}
For the Territory algorithm, $\mathcal{Q}_t^k$ and $\mathcal{S}_t^k$ are still given
by \eqref{qtk} but with $I_{t, i}^k$ now given by
\begin{equation} \label{indexlevel1}
I_{t, i}^k = \operatorname*{arg\,max}_{\ell \in \mathcal{S}_t^{k-1} \cup \{k\}} \mathcal{C}_t^{\ell}( x_{t-1}^i ),
\end{equation}
starting from $\mathcal{S}_t^{0}=\{0\}$. 

\subsubsection{Limited memory Level 1}

Note that if several cuts are the highest at the same trial point and in particular if they are identical
(which can happen in practice, see the numerical experiments of Section \ref{numsim}) then Level 1 will select all of them.
This leads us to propose a limited memory variant of this cut selection method.\\
\par {\textbf{Limited memory Level 1 cut selection}}: with the notation of the previous section, at iteration $k$, after computing the
cut for $\mathcal{Q}_t$, we store in $I_{t, i}^k,\,i=1,\ldots,k$, the index of {\textbf{only one}} of the highest cuts 
at $x_{t-1}^i$, namely among the cuts computed up to iteration $k$ that have the highest value at $x_{t-1}^i$,
we store the oldest cut, i.e., the cut that was first computed among the cuts that have the highest value at that point.

The pseudo-code for selecting the cuts for $\mathcal{Q}_t$ at iteration $k$ using this limited
memory variant of Level 1 is given in Figure \ref{figurecut1}.
In this pseudo-code, we use 
the notation $I_{t, i}$ in place of $I_{t, i}^k$. We also store in variable
$m_{t, i}$ the current value of the highest cut for $\mathcal{Q}_t$
at $x_{t-1}^{i}$. At the end of the first iteration, we initialize
$m_{t, 1} = \mathcal{C}_t^1( x_{t-1}^{1})$. After cut $\mathcal{C}_t^k$ is computed at 
iteration $k \geq 2$, these variables are updated using the relations
$$   
\left\{
\begin{array}{lll}
m_{t, i} &  \leftarrow & \max\Big(m_{t, i}, \mathcal{C}_t^k( x_{t-1}^{i})\Big),\;i=1,\ldots,k-1,\\
m_{t, k} &  \leftarrow & \max\Big(\mathcal{C}_t^j( x_{t-1}^{k}), j=1,\ldots,k \Big).
\end{array}
\right. 
$$
Finally, we use an array {\tt{Selected}} of Boolean using the information given by variables $I_{t, i}$
whose $i$-th entry is {\tt{True}} if cut $i$ is selected and {\tt{False}} otherwise. 
The index set $\mathcal{S}_t^k$ is updated correspondingly. Though all cuts are stored, only the selected cuts are used in the problems solved in the forward pass.\\

\begin{figure}
\rule{\linewidth}{1pt}\\
\flushleft
\begin{tabular}{l}
$I_{t, k}=\{1\}$, $m_{t, k}=\mathcal{C}_t^1( x_{t-1}^k )$.\\
{\textbf{For}} $\ell=1,\ldots,k-1$,\\
\hspace*{0.4cm}{\textbf{If }}$\mathcal{C}_t^k( x_{t-1}^{\ell} ) > m_{t, \ell}$ {\textbf{then }}$I_{t, \ell}=\{k\},\; m_{t, \ell}=\mathcal{C}_t^k( x_{t-1}^{\ell} )${\;\textbf{End If}}\\
\hspace*{0.4cm}{\textbf{If }}$\mathcal{C}_t^{\ell + 1}( x_{t-1}^{k} ) > m_{t, k}$ {\textbf{then }}$I_{t, k}=\{\ell + 1\},\; m_{t, k}=\mathcal{C}_t^{\ell + 1}( x_{t-1}^{k} )${\;\textbf{End If}}\\
{\textbf{End For}}\\
{\textbf{For}} $\ell=1,\ldots,k$,\\
\hspace*{0.7cm}{\tt{Selected}}[$\ell$]={\tt{False}}\\
{\textbf{End For}}\\
{\textbf{For}} $\ell=1,\ldots,k$\\
\hspace*{0.4cm}{\textbf{For}} $j=1,\ldots,|I_{t, \ell}|$\\
\hspace*{1cm} {\tt{Selected}}[$I_{t, \ell}[j]$]={\tt{True}}\\
\hspace*{0.4cm}{\textbf{End For}}\\
{\textbf{End For}}\\
$\mathcal{S}_t^k=\emptyset$\\
{\textbf{For}} $\ell=1,\ldots,k$\\
\hspace*{0.7cm}{\textbf{If}} {\tt{Selected}}[$\ell$]={\tt{True}} {\textbf{then}} $\mathcal{S}_t^k=\mathcal{S}_t^k \cup \{\ell\}$ {\textbf{End If}}\\
{\textbf{End For}}\\
\end{tabular}
\rule{\linewidth}{1pt}
\caption{Pseudo-code for limited memory Level 1 cut selection for DDP.}
\label{figurecut1}
\end{figure}

\if{
\begin{figure}
\begin{tabular}{l}
\input{Cut_Selection_SDDP.pstex_t}
\end{tabular}
\caption{A cut selection strategy for DDP (Level 1 or Territory algorithm). Red balls represent $\mathcal{C}_t^k( x_{t-1}^k )$.}
\label{figintro}
\end{figure}
This cut selection strategy is illustrated in Figure \ref{figintro}. 
In this figure, for the first three iterations, all the cuts have the largest value for at least
one trial point ($\mathcal{C}_t^1$ at $x_{t-1}^1$, $\mathcal{C}_t^2$ at $x_{t-1}^2$, and $\mathcal{C}_t^3$ at $x_{t-1}^3$,) so no cut is eliminated.
However, for the fourth iteration, the value of the cut $\mathcal{C}_t^2$
is strictly less than the value of the highest cut 
for all the trial points $x_{t-1}^1, x_{t-1}^2, x_{t-1}^3$, and $x_{t-1}^4$, so this cut is eliminated.
As pointed out in \cite{pfeifferetalcuts} and \cite{dpcuts0} and as we see on this example, a useful cut, i.e., a cut that is above all the others on some subset
can be eliminated (such is the case of cut $\mathcal{C}_t^2$).
Also, we do not necessarily have $\mathcal{Q}_t^k \geq \mathcal{Q}_t^j$ for every $j \in \mathcal{S}_t^k$.
 Indeed, on this example, $\mathcal{Q}_t^4 =\max(\mathcal{C}_t^1, \mathcal{C}_t^3, \mathcal{C}_t^4)$,
 $\mathcal{S}_t^4=\{1, 3, 4\}$,
 but we do not have $\mathcal{Q}_t^4  \geq \mathcal{Q}_t^3$ (on the set $\mathcal{S}$ represented in green in Figure \ref{figintro}).
 However, all redundant cuts are eliminated.
 
We now describe more precisely these algorithms in pseudo-code, starting with the
Level 1 cut selection  strategy.
For implementation purposes, it is convenient to introduce
variables $m_{t, j}$  storing the largest values of the cuts at the trial points:
these variables $m_{t, j}$ are defined for $t=2,\ldots,T$,
and $j=1,2,\ldots$ and $m_{t, j}$ is the current value of the highest cut for $\mathcal{Q}_t$
at trial point  $x_{t-1}^{j}$. At the end of the first iteration, we initialize
$m_{t, 1} = \mathcal{C}_t^1( x_{t-1}^{1})$. After cut $\mathcal{C}_t^k$ is computed at 
iteration $k \geq 2$, these variables are then updated using the relations
$$   
\left\{
\begin{array}{lll}
m_{t, j} &  \leftarrow & \max\Big(m_{t, j}, \mathcal{C}_t^k( x_{t-1}^{j})\Big),\;j=1,\ldots,k-1,\\
m_{t, k} &  = & \max\Big(\mathcal{C}_t^j( x_{t-1}^{k}), j=1,\ldots,k \Big),
\end{array}
\right. 
$$
We also introduce variables $I_{t, j}$ defined for $t=2,\ldots,T$,
and $j=1,2,\ldots$: $I_{t, j}$ is the list of cut indexes for $\mathcal{Q}_{t-1}$ which are the highest at $x_{t-1}^j$.
It is updated at each trial point where the last computed cut has the highest value.
Finally, we use an array {\tt{Selected}} of Boolean using the information given by variables $I_{t, j}$
whose $i$-th entry is {\tt{True}} if cut $i$ is selected and {\tt{False}} otherwise. 
The index set $\mathcal{S}_t^k$ is then updated correspondingly.
Summarizing, we obtain the algorithm given in Figure \ref{figurecut1} to update the cuts for $\mathcal{Q}_t$ at iteration $k$.
\begin{figure}
\rule{\linewidth}{1pt}\\
\flushleft
\begin{tabular}{l}
$I_{t, k}=\{1\}$, $m_{t, k}=\mathcal{C}_t^1( x_{t-1}^k )$.\\
{\textbf{For}} $\ell=1,\ldots,k-1$,\\
\hspace*{0.4cm}{\textbf{If }}$\mathcal{C}_t^k( x_{t-1}^{\ell} ) = m_{t, \ell}$ {\textbf{then }}$I_{t, \ell}=I_{t, \ell} \cup \{k\}$\\
\hspace*{0.4cm}{\textbf{Else if }}$\mathcal{C}_t^k( x_{t-1}^{\ell} ) > m_{t, \ell}$ {\textbf{then }}$I_{t, \ell}=\{k\},\; m_{t, \ell}=\mathcal{C}_t^k( x_{t-1}^{\ell} )$\\
\hspace*{0.4cm}{\textbf{End If}}\\
\hspace*{0.4cm}{\textbf{If }}$\mathcal{C}_t^{\ell + 1}( x_{t-1}^{k} ) = m_{t, k}$ {\textbf{then }}$I_{t, k}=I_{t, k} \cup \{\ell + 1\}$\\
\hspace*{0.4cm}{\textbf{Else if }}$\mathcal{C}_t^{\ell + 1}( x_{t-1}^{k} ) > m_{t, k}$ {\textbf{then }}$I_{t, k}=\{\ell + 1\},\; m_{t, k}=\mathcal{C}_t^{\ell + 1}( x_{t-1}^{k} )$\\
\hspace*{0.4cm}{\textbf{End If}}\\
{\textbf{End For}}\\
{\textbf{For}} $\ell=1,\ldots,k$,\\
\hspace*{0.7cm}{\tt{Selected}}[$\ell$]={\tt{False}}\\
{\textbf{End For}}\\
{\textbf{For}} $\ell=1,\ldots,k$\\
\hspace*{0.4cm}{\textbf{For}} $j=1,\ldots,|I_{t, \ell}|$\\
\hspace*{1cm} {\tt{Selected}}[$I_{t, \ell}[j]$]={\tt{True}}\\
\hspace*{0.4cm}{\textbf{End For}}\\
{\textbf{End For}}\\
$\mathcal{S}_t^k=\emptyset$\\
{\textbf{For}} $\ell=1,\ldots,k$\\
\hspace*{0.7cm}{\textbf{If}} {\tt{Selected}}[$\ell$]={\tt{True}} {\textbf{then}} $\mathcal{S}_t^k=\mathcal{S}_t^k \cup \{\ell\}$ {\textbf{End If}}\\
{\textbf{End For}}\\
\end{tabular}
\rule{\linewidth}{1pt}
\caption{Level 1 cut selection for iteration $k$ of DDP.}
\label{figurecut1}
\end{figure}
\par For the Territory algorithm, we use an additional variable {\tt{Nb\_Cuts}} which corresponds,
for a given step, to the number of current cuts for the approximate cost-to-function of that step.
Since the nonselected cuts are pruned, the set of stored cuts needs to be updated but 
the index set $\mathcal{S}_t^k$ is always $\mathcal{S}_t^k=\{1,2,\ldots,{\tt{Nb\_Cuts}}\}$ (all stored cuts are relevant).
The cut pruning algorithm for the Territory algorithm is given in Figure \ref{figurecut2}.\\
\begin{figure}[htbp]
\rule{\linewidth}{1pt}
\flushleft
\begin{tabular}{l}
{\tt{Nb\_Cuts}}={\tt{Nb\_Cuts}}+1. \hspace*{0.3cm}{\tt{//New cut $\mathcal{C}_t^k$ has just been computed}}\\
$\mathcal{C}_t^{{\tt{Nb\_Cuts}}} \leftarrow \mathcal{C}_t^k$.\\
$I_{t, k}=\{1\}$, $m_{t, k}=\mathcal{C}_t^1( x_{t-1}^{k} )$.\\
{\textbf{For}} $\ell=1,\ldots,k-1$,\\
\hspace*{0.4cm}{\textbf{If }}$\mathcal{C}_t^{\tt{Nb\_Cuts}}( x_{t-1}^{\ell} ) = m_{t, \ell}$ {\textbf{then }}$I_{t, \ell}=I_{t, \ell} \cup \{{\tt{Nb\_Cuts}}\}$\\
\hspace*{0.4cm}{\textbf{Else if }}$\mathcal{C}_t^{\tt{Nb\_Cuts}}( x_{t-1}^{\ell} ) > m_{t, \ell}$ {\textbf{then }}$I_{t, \ell}=\{{\tt{Nb\_Cuts}}\},\; m_{t, \ell}=\mathcal{C}_t^{\tt{Nb\_Cuts}}( x_{t-1}^{\ell} )$\\
\hspace*{0.4cm}{\textbf{End If}}\\
{\textbf{End For}}\\
{\textbf{For}} $\ell=1,\ldots,{\tt{Nb\_Cuts}}-1$,\\
\hspace*{0.4cm}{\textbf{If }}$\mathcal{C}_t^{\ell + 1}( x_{t-1}^{k} ) = m_{t, k}$ {\textbf{then }}$I_{t, k}=I_{t, k} \cup \{\ell + 1\}$\\
\hspace*{0.4cm}{\textbf{Else if }}$\mathcal{C}_t^{\ell + 1}( x_{t-1}^{k} ) > m_{t, k}$ {\textbf{then }}\\
\hspace*{0.8cm}$I_{t, k}=\{\ell + 1\},\; m_{t, k}=\mathcal{C}_t^{\ell + 1}( x_{t-1}^{k} )$\\
\hspace*{0.4cm}{\textbf{End If}}\\
{\textbf{End For}}\\
{\textbf{For}} $\ell=1,\ldots,{\tt{Nb\_Cuts}}$\\
\hspace*{0.7cm}{\tt{Selected}}[$\ell$]={\tt{False}}\\
{\textbf{End For}}\\
{\textbf{For}} $\ell=1,\ldots,k$\\
\hspace*{0.4cm}{\textbf{For}} $j=1,\ldots,|I_{t, \ell}|$\\
\hspace*{1cm} {\tt{Selected}}[$I_{t, \ell}[j]$]={\tt{True}}\\
\hspace*{0.4cm}{\textbf{End For}}\\
{\textbf{End For}}\\
$\mathcal{S}_t^k=\emptyset$.\\
{\tt{Count}}=1. \hspace*{0.2cm}{\tt{//Variable Count stores the number of relevant cuts}} \\
{\textbf{For}} $\ell=1,\ldots,{\tt{Nb\_Cuts}}$\\
\hspace*{0.4cm}{\textbf{If}} {\tt{Selected}}[$\ell$]={\tt{True}} {\textbf{then}}\\
\hspace*{0.8cm}$\mathcal{S}_t^k=\mathcal{S}_t^k \cup \{ {\tt{Count}} \}$\\
\hspace*{0.8cm}$\mathcal{C}_t^{{\tt{Count}}} \leftarrow \mathcal{C}_t^{\ell}$\hspace*{0.2cm}{\tt{//Cuts are re-ordered, nonrelevant cuts are pruned}}\\
\hspace*{0.8cm}{\textbf{For}} $i=1,\ldots,k$\\
\hspace*{1.2cm}{\textbf{For}} $j=1,\ldots,|I_{t, i}|$\\
\hspace*{1.9cm}{\textbf{If}} $I_{t, i}[j]= \ell$ {\textbf{then}} $I_{t, i}[j]={\tt{Count}}$ {\textbf{End If}}\\
\hspace*{1.2cm}{\textbf{End For}}\\
\hspace*{0.8cm}{\textbf{End For}}\\
\hspace*{0.8cm}{\tt{Count}}={\tt{Count}}+1 \\
\hspace*{0.4cm}{\textbf{End If}}\\
{\textbf{End For}}\\
{\tt{Nb\_Cuts}}={\tt{Count}}-1.
\end{tabular}
\rule{\linewidth}{1pt}
\caption{Territory algorithm for iteration $k$ of DDP.}
\label{figurecut2}
\end{figure}
}\fi

\par {\textbf{Discussion.} With these cut selection strategies, we do not have anymore $\mathcal{Q}_{t}^{k} \geq \mathcal{Q}_{t}^{k-1}$ for all iterations $k$.
However, the relation $\mathcal{Q}_{t}^{k} \geq \mathcal{Q}_{t}^{k-1}$ holds for iterations $k$ such that,
at the previous trial points $x_{t-1}^\ell, \ell=1,\ldots,k-1$,
the value of the new cut $\mathcal{C}_t^k$ is below the value of
$\mathcal{Q}_{t}^{k-1}$: 
\begin{equation}\label{condcroiss}
\mathcal{C}_t^k( x_{t-1}^\ell ) \leq \mathcal{Q}_{t}^{k-1}(x_{t-1}^\ell), \ell=1,\ldots,k-1.
\end{equation}
\if{
In this case, for the Territory algorithm, two situations, represented in Figure \ref{figpartcaseqt}, can happen:
\begin{itemize}
\item[(i)] one of the inequalities \eqref{condcroiss} is an equality or
$\mathcal{C}_t^k( x_{t-1}^k ) \geq \mathcal{Q}_{t}^{k-1}(x_{t-1}^k)$. In this case, $\mathcal{Q}_t^{k}= \max(\mathcal{Q}_t^{k-1}, \mathcal{C}_t^k)$
(top right figure in Figure \ref{figpartcaseqt}).
\item[(ii)] $\mathcal{C}_t^k( x_{t-1}^k )< \mathcal{Q}_{t}^{k-1}(x_{t-1}^k)$ and all the inequalities \eqref{condcroiss} are strict, in which case
$\mathcal{Q}_t^{k}= \mathcal{Q}_t^{k-1}$ (top left figure in Figure \ref{figpartcaseqt}).
\end{itemize}
For the Level 1 cut selection strategy, if \eqref{condcroiss} holds, then:
\begin{itemize}
\item[(a)] if $\mathcal{C}_t^k( x_{t-1}^k )< \mathcal{Q}_{t}^{k-1}(x_{t-1}^k)$,
if all inequalities \eqref{condcroiss} are strict, and if 
$\mathcal{C}_t^j( x_{t-1}^k )< \mathcal{Q}_{t}^{k-1}(x_{t-1}^k)$ for all $j \in \{1,2,\ldots,k-1\} \backslash S_t^{k-1}$,
then $\mathcal{Q}_t^{k}= \mathcal{Q}_t^{k-1}$ (middle left figure in Figure \ref{figpartcaseqt}).
\item[(b)] If one of the inequalities \eqref{condcroiss} is an equality or 
if ($k\in I_{t, k}^k$ and $I_{t, k}^k \subseteq \{k\} \cup  \mathcal{S}_{t}^{k-1}$) then 
$\mathcal{Q}_t^{k}= \max(\mathcal{Q}_t^{k-1}, \mathcal{C}_t^k)$ (middle right figure in Figure \ref{figpartcaseqt}).
\item[(c)] If all inequalities \eqref{condcroiss} are strict and
$I_{t, k}^k \subseteq \{1,2,\ldots,k-1\} \backslash S_t^{k-1}$ then
$\mathcal{Q}_t^{k}= \max(\mathcal{Q}_t^{k-1}, \displaystyle \max_{j \in I_{t, k}^{k}} \mathcal{C}_t^j )$
(bottom left figure in Figure \ref{figpartcaseqt}).
\item[(d)] If (one of the inequalities \eqref{condcroiss} is an equality or $k\in I_{t, k}^k$) and 
($I_{t, k}^{k} \backslash \{k\} \cap \mathcal{S}_t^{k-1} \neq \emptyset$) then 
$I_{t, k}^{k} \backslash \{k\} \neq \emptyset$,
$\mathcal{Q}_t^{k}= \max(\mathcal{Q}_t^{k-1}, \mathcal{C}_t^k, \displaystyle \max_{j \in I_{t, k}^{k} \backslash \{k\}} \mathcal{C}_t^j )$
(bottom right figure in Figure \ref{figpartcaseqt}).
\end{itemize}
\begin{figure}
\begin{tabular}{ll}
\input{NewCut0T.pstex_t} & \input{NewCutkT.pstex_t} \\
\input{NewCut0.pstex_t} & \input{NewCutkLev1.pstex_t}\\
\input{NewCutjLev1.pstex_t} & \input{NewCutjkLev1.pstex_t}
\end{tabular}
\caption{Six cases where $\mathcal{Q}_t^{k} \geq  \mathcal{Q}_t^{k-1}$ (it is assumed that relations \eqref{condcroiss} hold).}
\label{figpartcaseqt}
\end{figure}
}\fi
For the Territory algorithm, Level 1, and its limited memory variant, we have that
\begin{equation}\label{ineq1}
\mathcal{Q}_t^k ( x_{t-1}^k  ) \geq  \theta_t^k = \mathcal{C}_t^k ( x_{t-1}^k  ),\;\; \forall t=2,\ldots,T, \;\forall k \in \mathbb{N}^{*},
\end{equation}
and by definition of the cut selection strategies we have for $k_2 \geq k_1$ and $j=1,\ldots,k_2$, that
$$
\mathcal{Q}_t^{k_2} ( x_{t-1}^{j}  ) = \displaystyle \max_{\ell=1,\ldots,k_2} \mathcal{C}_t^{\ell} \Big( x_{t-1}^{j} \Big),
$$
which implies that for all $k_2 \geq k_1$ and for $j=1,\ldots,k_2$,
\begin{equation}\label{ineq2}
\mathcal{Q}_t^{k_2} ( x_{t-1}^{j}  ) \geq \displaystyle \max_{\ell=1,\ldots,k_1} \mathcal{C}_t^{\ell} \Big( x_{t-1}^{j} \Big) = \mathcal{Q}_t^{k_1} ( x_{t-1}^{j}  ),
\end{equation}
i.e., Assumption (H2) is satisfied.
\begin{rem} Relation \eqref{ineq1} means that whenever a new trial point is generated, the value of the approximate cost-to-go function 
at this point is a better (larger) approximation of the value of the cost-to-go function at this point
than the value of the last computed cut at this point. Relations \eqref{ineq2} show
that at the trial points, the approximations of the cost-to-go functions tend to be better along the iterations, i.e., for every $k$
the sequence $\mathcal{Q}_t^{j} ( x_{t-1}^{k}  )_{j \geq k-1}$ is nondecreasing.
\end{rem}

\section{Convergence analysis for nonlinear problems} \label{convanalysis}

To proceed, we need the following lemma:
\begin{lemma}\label{contqtk}
Consider optimization problem \eqref{optconvexd},
recourse functions $\mathcal{Q}_t, t=1,\ldots,T+1$, given by 
\eqref{dpconvexd}, and let Assumption (H1) hold. 
To solve that problem, consider the DDP algorithm with cut selection presented in the previous
section.
Then
for $t=2,\ldots,T+1$ and all $k \in \mathbb{N}$, $\mathcal{Q}_t^k$ is convex
and for all $k \geq T-t+1$, $\mathcal{Q}_t^k$, is Lipschitz continuous on 
$\mathcal{X}_{t-1}^{\varepsilon}$, and coefficients $\theta_t^k$ and $\beta_t^k$ are bounded.
\end{lemma}
\begin{proof} See the proof of Lemma 3.2 in \cite{guigues2014cvsddp}.\hfill
\end{proof}

The convergence of DDP with cut selection for nonlinear programs is given in Theorem \ref{convddpth}
below. In particular, introducing at iteration $k$ the approximation ${\underline{\mathcal{Q}}}_k = \sup_{n \leq k} {\underline{\mathcal{Q}}}_1^n( x_0 )$
(which can be computed at iteration $k$) of $\mathcal{Q}_1( x_0)$, we show that the whole sequence ${\underline{\mathcal{Q}}}_k$
converges to $\mathcal{Q}_1( x_0 )$. This proof relies on the following properties:
\begin{enumerate}
\item decisions belong to compact sets ((H1)-(a));
\item for $t=1,\ldots,T+1$, $\mathcal{Q}_t$ is convex and continuous on $\mathcal{X}_{t-1}$
(Lemma \ref{contqt});
\item for all $t=2,\ldots,T+1$ and all $k \in \mathbb{N}$, $\mathcal{Q}_t^k$ is convex
and for all $k \geq T-t+1$, $\mathcal{Q}_t^k$, is Lipschitz continuous on 
$\mathcal{X}_{t-1}$, and coefficients $\theta_t^k$ and $\beta_t^k$ are bounded (Lemma \ref{contqtk}).
\end{enumerate}
\begin{thm}[Convergence of DDP with cut selection for nonlinear programs.] \label{convddpth}
To solve problem \eqref{optconvexd}, consider the DDP algorithm 
presented in Section \ref{algorithmddp} combined with a cut selection strategy 
satisfying Assumption (H2).
Consider the sequences of vectors $x_t^k$ and functions $\mathcal{Q}_t^k$
generated by this algorithm.
Let Assumption (H1) hold. 
Let $(x_1^*,\ldots,x_T^*) \in \mathcal{X}_1\small{\times}\cdots \small{\times} \mathcal{X}_T$ be an accumulation point of the sequence
$(x_1^k, \ldots, x_T^k )_{k \geq 1}$ and let $K$ be an infinite subset of integers such that
for $t=1,\ldots,T$, $\lim_{k \to +\infty,\, k \in K} x_t^k = x_{t}^{*}$.
Then 
\begin{itemize}
\item[(i)] for all $k \geq 1$ we have $\mathcal{Q}_{T+1}( x_T^k)=\mathcal{Q}_{T+1}^k( x_T^k)$,
\begin{equation}\label{equalqT}
\mathcal{Q}_{T}( x_{T-1}^k)= \mathcal{Q}_{T}^k( x_{T-1}^k ) = {\underline{\mathcal{Q}}}_{T}^k( x_{T-1}^k), 
\end{equation}
and for $t=2,\ldots,T-1$,
$$
H(t): \;\displaystyle \lim_{k \to +\infty,\, k \in K} \mathcal{Q}_{t}^{k}(x_{t-1}^{k}) =\displaystyle \lim_{k \to +\infty,\, k \in K} {\underline{\mathcal{Q}}}_{t}^{k}(x_{t-1}^{k}) = \mathcal{Q}_{t}(x_{t-1}^{*}).
$$
\item[(ii)] Setting ${\underline{\mathcal{Q}}}_k = \sup_{n \leq k} {\underline{\mathcal{Q}}}_1^n( x_0 )$, we have
$
\displaystyle \lim_{k \to \infty} {\underline{\mathcal{Q}}}_k  = \mathcal{Q}_{1}(x_0)
$ 
and $(x_1^*,\ldots,x_T^*)$ is an optimal solution of \eqref{optconvexd}.
\end{itemize}
\end{thm}
\begin{proof} First note that the existence of an accumulation point comes from the fact
that the sequence $(x_1^k, x_2^k, \ldots, x_T^k)_{k \geq 1}$ is a sequence
of the compact set $\mathcal{X}_1\small{\times}\cdots \small{\times} \mathcal{X}_T$.
\par (i) Since $\mathcal{Q}_{T+1}=\mathcal{Q}_{T+1}^k =0$,
we have $\mathcal{Q}_{T+1}( x_T^k)=\mathcal{Q}_{T+1}^k( x_T^k)=0$
and by definition of $\mathcal{Q}_{T}, {\underline{\mathcal{Q}}}_{T}^k$, we also
have $\mathcal{Q}_{T} = {\underline{\mathcal{Q}}}_{T}^k$.
Next, recall that $\mathcal{Q}_{T}^k ( x_{T-1}^k  )  \leq \mathcal{Q}_{T}( x_{T-1}^k )$
and using Assumption (H2), we have
$$
\mathcal{Q}_{T}^k ( x_{T-1}^k  )  \geq \mathcal{C}_T^k (  x_{T-1}^k ) = \theta_T^k  = {\underline{\mathcal{Q}}}_T^k (  x_{T-1}^k )= \mathcal{Q}_{T}( x_{T-1}^k ).
$$
We have thus shown \eqref{equalqT}.

We show $H(t), t=2,\ldots,T-1$, by backward induction on $t$. Due to \eqref{equalqT} and the continuity of $\mathcal{Q}_T$, we have that $H(T)$ hold.
Now assume that $H(t+1)$ holds for some $t \in \{2,\ldots,T-1\}$. We want to show that $H(t)$ holds.
We have for $k \in K$ that
\begin{equation} \label{eqconv1}
\begin{array}{lll}
\mathcal{Q}_{t}(x_{t-1}^k) \geq  \mathcal{Q}_{t}^k(x_{t-1}^k) & \geq & \theta_t^k \mbox{ using }\mbox{(H2)},\\
& \geq &  f_t(x_t^ k, x_{t-1}^ k) + \mathcal{Q}_{t+1}^{k-1}(x_t^k) \mbox{ by definition of }\theta_t^k,\\
& \geq &  F_t(x_t^ k, x_{t-1}^ k) - \mathcal{Q}_{t+1}(x_t^k)  + \mathcal{Q}_{t+1}^{k-1}(x_t^k)\mbox{ by definition of }F_t,\\
& \geq & \mathcal{Q}_t(x_{t-1}^ k) - \mathcal{Q}_{t+1}(x_t^k)  + \mathcal{Q}_{t+1}^{k-1}(x_t^k),
\end{array}
\end{equation}
where the last inequality comes from the fact that $x_t^k$ is feasible for the problem
defining $\mathcal{Q}_t(x_{t-1}^k)$.

Let $K=\{y_k \;|\; k \in \mathbb{N}\}$.
Denoting $\ell=\lim_{k \to +\infty} \mathcal{Q}_{t+1}^{y_k}(x_t^{y_k})$,
we now show that $$\ell=\lim_{k \to +\infty} \mathcal{Q}_{t+1}^{y_k - 1}(x_t^{y_k}).$$
Fix $\varepsilon_0>0$.
From Lemma \ref{contqtk}, functions $\mathcal{Q}_{t+1}^k, k \geq T$, are
Lipschitz continuous on $\mathcal{X}_t$.
Let $L$ be a Lipschitz constant for these functions, independent on $k$ (see \cite{guigues2014cvsddp}
for an expression of $L$, independent on $k$, expressed in terms of the problem data). 
Since $\lim_{k \to +\infty} \;x_{t}^{y_k}=x_{t}^{*}$, 
there exists $k_0 \in \mathbb{N}$ with $k_0\geq T$ such that for $k \geq k_0$, we have
\begin{equation} \label{conxtmL}
\|x_{t}^{y_k} - x_{t}^{*}  \| \leq \frac{\varepsilon_0}{4L} \mbox{ and }
\ell - \frac{\varepsilon_0}{2} \leq \mathcal{Q}_{t+1}^{y_k}(x_t^{y_k}) \leq \ell + \frac{\varepsilon_0}{2}.
\end{equation}
Take now $k \geq k_0+1$. Since $y_k \geq y_{k_0} +1$, using Assumption (H2) we have
\begin{equation}\label{lem1eq}
\mathcal{Q}_{t+1}^{y_k -1}(x_t^{y_{k_0}}) \geq  \mathcal{Q}_{t+1}^{y_{k_0}}(x_t^{y_{k_0}}).
\end{equation}
Using Assumption (H2) again, we obtain
\begin{equation}\label{lem1eq2}
\mathcal{Q}_{t+1}^{y_k}(x_t^{y_k})  \geq \mathcal{Q}_{t+1}^{y_k - 1}(x_t^{y_k}).
\end{equation}
Recalling that $k \geq k_0+1$ and that $y_k - 1 \geq y_{k_0} \geq k_0 \geq T$, observe that 
\begin{equation}\label{lem1eq3}
| \mathcal{Q}_{t+1}^{y_k -1}(x_t^{y_k}) - \mathcal{Q}_{t+1}^{y_k -1}(x_t^{y_{k_0}}) | \leq L \|x_t^{y_{k}}-x_t^{y_{k_0}}\| \stackrel{\eqref{conxtmL}}{\leq} \frac{\varepsilon_0}{2}.
\end{equation}
It follows that 
\begin{equation}\label{firstdiff}
 \begin{array}{lll}
\max\Big(\mathcal{Q}_{t+1}^{y_k -1}(x_t^{y_{k_0}}),\mathcal{Q}_{t+1}^{y_k -1}(x_t^{y_k})\Big) &\stackrel{\eqref{lem1eq3}}{\leq}&
\min\Big(\mathcal{Q}_{t+1}^{y_k -1}(x_t^{y_{k_0}}),\mathcal{Q}_{t+1}^{y_k -1}(x_t^{y_k})\Big)  + \frac{\varepsilon_0}{2} \\
&\stackrel{\eqref{lem1eq2}}{\leq}& \mathcal{Q}_{t+1}^{y_k}(x_t^{y_k})  + \frac{\varepsilon_0}{2} \\
&\stackrel{\eqref{conxtmL}}{\leq}&\ell + \varepsilon_0.
\end{array}
\end{equation}
We also have 
\begin{equation} \label{secondiff}
\begin{array}{lcl}
\max\Big(\mathcal{Q}_{t+1}^{y_k -1}(x_t^{y_{k_0}}),\mathcal{Q}_{t+1}^{y_k -1}(x_t^{y_k})\Big) 
&\geq&  \min\Big(\mathcal{Q}_{t+1}^{y_k -1}(x_t^{y_{k_0}}),\mathcal{Q}_{t+1}^{y_k -1}(x_t^{y_k})\Big) \\
&\stackrel{\eqref{lem1eq3}}{\geq}&
\max\Big(\mathcal{Q}_{t+1}^{y_k -1}(x_t^{y_{k_0}}),\mathcal{Q}_{t+1}^{y_k -1}(x_t^{y_k})\Big)  - \frac{\varepsilon_0}{2} \\
&\stackrel{\eqref{lem1eq}}{\geq}& \mathcal{Q}_{t+1}^{y_{k_0}}(x_t^{y_{k_0}})  -\frac{\varepsilon_0}{2} \\
&\stackrel{\eqref{conxtmL}}{\geq} &\ell - \varepsilon_{0}.
\end{array}
\end{equation}
We have thus shown that both $\mathcal{Q}_{t+1}^{y_k -1}(x_t^{y_{k_0}})$ and $\mathcal{Q}_{t+1}^{y_k -1}(x_t^{y_k})$ belong to the interval 
$[\ell-\varepsilon_{0}, \ell+\varepsilon_{0}]$.
Summarizing, we have shown that for every $\varepsilon_{0}>0$, there exists $k_0$ such that
for all $k \geq k_0+1$ we have
$$
\ell - \varepsilon_0 \leq \mathcal{Q}_{t+1}^{y_k -1}(x_t^{y_k})  \leq \ell + \varepsilon_{0},
$$
which shows that $\ell=\lim_{k \to +\infty} \mathcal{Q}_{t+1}^{y_k - 1}(x_t^{y_k})$.
It follows that 
\begin{equation} \label{eqconv3}
\lim_{k \to +\infty,\, k \in K} \mathcal{Q}_{t+1}^{k-1}(x_t^k)-\mathcal{Q}_{t+1}(x_t^k)=\lim_{k \to +\infty,\, k \in K} \mathcal{Q}_{t+1}^{k}(x_t^k)-\mathcal{Q}_{t+1}(x_t^k).
\end{equation}
Due to $H(t+1)$ and the continuity of 
$\mathcal{Q}_{t+1}$ on $\mathcal{X}_t$ (Lemma \ref{contqt}), the right hand side of \eqref{eqconv3}
is null. It follows that 
\begin{equation}\label{eqconv2}
\lim_{k \to +\infty,\, k \in K} -\mathcal{Q}_{t+1}(x_t^k)  + \mathcal{Q}_{t+1}^{k-1}(x_t^k) =0.
\end{equation}
Plugging \eqref{eqconv2} into \eqref{eqconv1}, recalling that $\theta_t^k= {\underline{\mathcal{Q}}}_{t}^{k}(x_{t-1}^{k})$,
using the fact that $\lim_{k \to +\infty,\, k \in K} x_{t-1}^k = x_{t-1}^{*}$
and the continuity of  $\mathcal{Q}_t$ (Lemma \ref{contqt}), we obtain $H(t)$.
\par (ii) By definition of ${\underline {Q}}_1^{k}$, we have
$$
{\underline{Q}}_1^{k} (x_{0})=F_1(x_1^{k}, x_0)-\mathcal{Q}_2(x_1^{k})+\mathcal{Q}_2^{k-1} (x_1^k)
$$
which implies 
\begin{equation}\label{finaldiii}
0 \leq - {\underline{Q}}_1^{k}(x_{0}) +\mathcal{Q}_1(x_{0})\leq \mathcal{Q}_2(x_1^k)-\mathcal{Q}_2^{k-1}(x_1^k).
\end{equation}
From (i), $H(2)$ holds, which, combined with the continuity
of $\mathcal{Q}_2$ (Lemma \ref{contqt}), gives 
$$\lim_{k \to +\infty,\,k\in K}\; \mathcal{Q}_{2}(x_1^k)-\mathcal{Q}_{2}^{k}(x_1^k)=0.
$$
Following the proof of (i), we show that this implies 
\begin{equation} \label{consdet}
\lim_{k \to +\infty,\,k\in K}\; \mathcal{Q}_{2}(x_1^k)-\mathcal{Q}_{2}^{k-1}(x_1^k)=0.
\end{equation}
Combining this relation with \eqref{finaldiii}, we obtain 
\begin{equation} \label{convvalopt}
\lim_{k \to +\infty,\,k\in K}\;{\underline{Q}}_1^{k} (x_{0}) =\mathcal{Q}_1(x_0).
\end{equation}

It follows that for every $\varepsilon_{0}>0$, there exists $k_0 \in K$ such that for $k \in K$
with $k \geq k_0$, we have $0 \leq \mathcal{Q}_1( x_0 ) - {\underline{\mathcal{Q}}}_1^k ( x_0 ) \leq \varepsilon_0$.
We then have for every $k \in \mathbb{N}$ with
$k \geq k_0$ that
$$
0 \leq \mathcal{Q}_1( x_0 ) - {\underline{\mathcal{Q}}}_k  \leq \mathcal{Q}_1( x_0 ) - {\underline{\mathcal{Q}}}_1^{k_0}( x_0 ) \leq \varepsilon_{0},
$$
which shows that $\lim_{k \to +\infty} {\underline{\mathcal{Q}}}_k = \mathcal{Q}_1( x_0 )$.

Take now $t \in \{1,\ldots,T\}$. By definition of $x_t^k$, we have
\begin{equation}\label{convfirststaged}
\begin{array}{lll}
 f_t(x_t^k, x_{t-1}^k) + \mathcal{Q}_{t+1}^{k-1}(x_t^k)&=& {\underline{Q}}_t^{k}(x_{t-1}^k).
\end{array} 
\end{equation}
Setting $x_0^* = x_0$, recall that we have shown that $\lim_{k \in K, k \to +\infty } {\underline{Q}}_t^{k}(x_{t-1}^k) = \mathcal{Q}_t (  x_{t-1}^* )$
(for $t=1$, this is \eqref{convvalopt}; for $t \in \{2, \ldots, T-1\}$ this is $H(t)$; and for 
$t=T$ this is obtained taking the limit in \eqref{equalqT} when $k \to +\infty$ with $k \in K$).

Taking the limit in \eqref{convfirststaged} when $k \to +\infty$ with $k \in K$, 
using \eqref{eqconv2}, the continuity of $\mathcal{Q}_{t+1}$ (Lemma \ref{contqt}) and
the lower semi-continuity of $f_t$, we obtain
\begin{equation}\label{x*opt}
F_t(x_{t}^*, x_{t-1}^* ) = f_t(x_{t}^* , x_{t-1}^* ) + \mathcal{Q}_{t+1}(x_t^* ) \leq 
\lim_{k \in K, k \to +\infty }  {\underline{Q}}_t^{k}(x_{t-1}^k)  =\mathcal{Q}_t ( x_{t-1}^* ).
\end{equation}
Passing to the limit in the relations $A_t x_t^k + B_t x_{t-1}^k = b_t$, we obtain
$A_t x_t^* + B_t x_{t-1}^* = b_t$. Since $g_t$ is lower semicontinuous, its level sets
are closed, which implies that $g_t( x_t^* , x_{t-1}^* ) \leq 0$. Recalling that
$x_t \in \mathcal{X}_t$, we have that $x_t^*$ is feasible for the problem
defining $\mathcal{Q}_t(  x_{t-1}^* )$.
Combining this observation with \eqref{x*opt}, we have shown that
$x_t^*$ is an optimal solution for the problem defining $\mathcal{Q}_t(  x_{t-1}^* )$, i.e., 
problem \eqref{dpconvexd} written for $x_{t-1}= x_{t-1}^*$. This shows that
$(x_1^*, \ldots, x_T^*)$ is an optimal solution to \eqref{optconvexd}. \hfill
\end{proof}

We now discuss the case when for some time steps the objective and/or the constraints are linear.
If for a given time step $t$, $\mathcal{X}_t$ is a polytope, and we 
do not have the nonlinear constraints $g_t(x_t, x_{t-1}) \leq 0$ (i.e., the constraints are linear), then
the conclusions
of Lemmas \ref{contqt}  and \ref{contqtk} hold and thus 
Theorem \ref{convddpth} holds too under weaker assumptions.
More precisely, for such time steps, we assume (H1)-(a), (H1)-(b), and instead of (H1)-(d), the weaker assumption:\\
\par (H1)-(d') there exists $\varepsilon>0$ such that 
$\mathcal{X}_{t}^{\varepsilon} {\small{\times}}  \mathcal{X}_{t-1} \subset \mbox{dom}(f_t)$
and for every $x_{t-1} \in \mathcal{X}_{t-1}$, there exists  
$x_t \in \mathcal{X}_t$ such that $\displaystyle A_{t} x_{t} + B_{t} x_{t-1} = b_t$.\\

\par If, additionally,  for a given time step $t$ the objective $f_t$ is linear, i.e., if both the objective and
the constraints are linear, then (H1)-(b) and the first part of 
(H1)-(d'), namely the existence of $\varepsilon>0$ such that 
$\mathcal{X}_{t}^{\varepsilon} {\small{\times}}  \mathcal{X}_{t-1} \subset \mbox{dom}(f_t)$, are automatically
satisfied. It follows that Theorem \ref{convddpth} still holds assuming (H1) for the time steps where
the constraints are nonlinear and (H1)-(a), (H1)-(b), (H1)-(d') for the time steps where the constraints
are linear, i.e., the time steps where $\mathcal{X}_t$ is a polytope and the
constraints $g_t(x_t, x_{t-1}) \leq 0$ are absent. 

\par However, in the special case of linear programs, we prove in Theorem \ref{convddpthlp} below that DDP with cut selection converges in a finite number of iterations

\section{Convergence analysis for linear problems} \label{convanalysislp}

The DDP algorithm can be applied to the following linear programs of form \eqref{optconvexd} without the nonlinear constraints $g_t(x_t, x_{t-1}) \leq 0$:
\begin{equation}\label{optlp}
\left\{
\begin{array}{l}
\displaystyle \inf_{x_1,\ldots,x_T}\;\sum_{t=1}^{T} f_t(x_{t}, x_{t-1}):=c_t^T x_t + d_t^T x_{t-1} \\
x_t \in \mathcal{X}_t,\; \displaystyle A_{t} x_{t} + B_{t} x_{t-1} = b_t, t=1,\ldots,T,
\end{array}
\right.
\end{equation}
where $\mathcal{X}_1,\ldots,\mathcal{X}_T$, are polytopes.
Making Assumptions (H1)-(a) and (H1)-(d')
for every time $t=1,\ldots,T$, Theorem \ref{convddpth} applies.
However, in this special case, we prove in Theorem \ref{convddpthlp} that DDP converges in a finite number of iterations
to an optimal solution. The proof of this theorem is similar to the proof of Lemmas 1 and 4 in \cite{philpot}.
Observe also that for problems of form \eqref{optlp}, Assumptions (H1)-(a) and (H1)-(d') can be stated as follows:
$\mathcal{X}_t$ is a nonempty, convex, compact polytope and for every $x_{t-1} \in \mathcal{X}_{t-1}$, there exists  
$x_t \in \mathcal{X}_t$ such that $\displaystyle A_{t} x_{t} + B_{t} x_{t-1} = b_t$.
\begin{thm}[Convergence of DDP with cut selection for linear programs.] \label{convddpthlp}
To solve problem \eqref{optlp}, consider the DDP algorithm combined with 
the Territory algorithm, Level 1 cut selection, or limited memory  Level 1 cut selection.  
Consider the sequences of vectors $x_t^k$ and functions $\mathcal{Q}_t^k$
generated by this algorithm.
Setting $\mathcal{X}_0 = \{ x_0 \}$, assume that for $t=1,\ldots,T$,
$\mathcal{X}_t$ is a nonempty, convex, compact, polytope and that for every $x_{t-1} \in \mathcal{X}_{t-1}$, there exists  
$x_t \in \mathcal{X}_t$ such that $\displaystyle A_{t} x_{t} + B_{t} x_{t-1} = b_t$.
Assume also that 
\begin{itemize}
\item[(H3)] the linear subproblems solved along the iterations of DDP are solved using an
algorithm that necessarily outputs a vertex as an optimal solution.\footnote{For instance a simplex algorithm.} 
\end{itemize}
Then there exists $k_0 \in \mathbb{N}^*$ such that
\begin{itemize}
\item[(i)] for every $k \geq k_0-1$, for every $t=2,\ldots,T$, we have $\mathcal{Q}_t^k = \mathcal{Q}_t^{k_0 -1}$.
\item[(ii)] $(x_1^{k_0},\ldots,x_T^{k_0})$ is an optimal solution to \eqref{optlp}.
\end{itemize}
\end{thm}
\begin{proof}
\par (i) Using Assumption (H3), the algorithm can only generate a finite number of different trial points $x_t^k$.\footnote{Without this assumption,
if a face of a polyhedron is solution, the algorithm could potentially return an infinite number of trial points and cuts.}
Indeed, this is true
for $t=1$ (under Assumption (H3), $x_1^k$ is an extreme point of a polyhedron) and assuming that the result is true for $t<T$ then there is a finite number of polyhedrons, parametrized by $x_t^k$,
to which $x_{t+1}^k$ can belong and $x_{t+1}^k$ is an extreme point of one of these polyhedrons.
Reasoning as in the proof of Lemma 1 in \cite{philpot}, we also check
that the algorithm can only compute a finite number of different cuts. Once no new trial point is computed, a cut that is suppressed will never be selected
again (since it is already dominated at all the trial points generated by the algorithm).
It follows that after a given iteration $k_{0}-1$, no new
trial point and no new cut is computed. Since the cut selection strategy is uniquely defined by the history of cuts,
the cost-to-go functions stabilize (note that the final selected cuts can be different for Level 1 and its limited memory variant but for a fixed cut selection strategy, they are
uniquely defined by the set of trial points and cuts computed until iteration $k_0-1$).
\par (ii) We first show by backward induction that for $t=1,\ldots,T$, we have
$$
\mathcal{H}_2(t): \;\;{\underline{\mathcal{Q}}}_t^{k_0} ( x_{t-1}^{k_0} )= \mathcal{Q}_t ( x_{t-1}^{k_0}).
$$
Since ${\underline{\mathcal{Q}}}_T^{k_0} = \mathcal{Q}_T$, the relation holds for $t=T$. Assume now that the relation holds
for $t+1$ with $t \in \{1,\ldots,T-1\}$. By definition of $x_t^{k_0}$, we have
$$
{\underline{\mathcal{Q}}}_t^{k_0} ( x_{t-1}^{k_0} ) = f_t (x_t^{k_0}, x_{t-1}^{k_0} ) + \mathcal{Q}_{t+1}^{k_0 - 1} ( x_{t}^{k_0} ).
$$
Recall that $\mathcal{Q}_{t+1}^{k_0 - 1} ( x_{t}^{k_0} ) \leq \mathcal{Q}_{t+1} ( x_{t}^{k_0} )$. If this inequality was
strict, i.e., if we had
\begin{equation}
\mathcal{Q}_{t+1}^{k_0 - 1} ( x_{t}^{k_0} ) < \mathcal{Q}_{t+1} ( x_{t}^{k_0} ) 
\end{equation}
then using $\mathcal{H}_2(t+1)$ we would also have $\mathcal{Q}_{t+1}^{k_0 - 1} ( x_{t}^{k_0} ) < {\underline{\mathcal{Q}}}_{t+1}^{k_0} ( x_{t}^{k_0} )$
and at iteration $k_0$ we could add a new cut for $\mathcal{Q}_{t+1}$ at $x_{t}^{k_0}$ with height ${\underline{\mathcal{Q}}}_{t+1}^{k_0} ( x_{t}^{k_0} )$ at this point, which is in contradiction with (i).
Therefore we must have $\mathcal{Q}_{t+1}^{k_0 - 1} ( x_{t}^{k_0} ) = \mathcal{Q}_{t+1} ( x_{t}^{k_0} )$ which implies
${\underline{\mathcal{Q}}}_t^{k_0} ( x_{t-1}^{k_0}   ) = F_t( x_{t}^{k_0}, x_{t-1}^{k_0} ) \geq \mathcal{Q}_t (  x_{t-1}^{k_0}  )$.
Since we also have ${\underline{\mathcal{Q}}}_t^{k_0} \leq  \mathcal{Q}_t$, we have shown $\mathcal{H}_2(t)$. We now check that
for $t=1,\ldots,T$, we have
$$
\mathcal{H}_3(t):\;\mathcal{Q}_1( x_0 ) = \sum_{j=1}^t f_j( x_j^{k_0}, x_{j-1}^{k_0} ) + \mathcal{Q}_{t+1}( x_t^{k_0}), 
$$
recalling that $x_0^{k_0}=x_0$.
Indeed, we have 
$$
\mathcal{Q}_1( x_0 ) \stackrel{\mathcal{H}_2(1)}{=} {\underline{\mathcal{Q}}}_1^{k_0} ( x_0 ) = f_1 ( x_1^{k_0}, x_0 ) + \mathcal{Q}_2^{k_0 -1}( x_1^{k_0} ) =f_1 ( x_1^{k_0}, x_0 ) + \mathcal{Q}_2 ( x_1^{k_0} )
$$
which is $\mathcal{H}_3(1)$. Assuming that $\mathcal{H}_3(t)$ holds for $t<T$, we have
$$
\begin{array}{lll}
\mathcal{Q}_1( x_0 )& =& \sum_{j=1}^t f_j( x_j^{k_0}, x_{j-1}^{k_0} ) + \mathcal{Q}_{t+1}( x_t^{k_0}) \stackrel{\mathcal{H}_2(t+1)}{=} \sum_{j=1}^t f_j( x_j^{k_0}, x_{j-1}^{k_0} ) + {\underline{\mathcal{Q}}}_{t+1}^{k_0}( x_t^{k_0})\\
& =& \sum_{j=1}^t f_j( x_j^{k_0}, x_{j-1}^{k_0} ) + f_{t+1}( x_{t+1}^{k_0}, x_{t}^{k_0}  )  + \mathcal{Q}_{t+2}^{k_0 -1}( x_{t+1}^{k_0}),\;\ \mbox{by definition of }x_{t+1}^{k_0},\\
& =& \sum_{j=1}^{t+1} f_j( x_j^{k_0}, x_{j-1}^{k_0} )  + \mathcal{Q}_{t+2}( x_{t+1}^{k_0}),
\end{array}
$$
which shows $\mathcal{H}_3(t+1)$.
$\mathcal{H}_3(T)$ means than $(x_1^{k_0},\ldots,x_T^{k_0})$ is an optimal solution to \eqref{optlp}.
\end{proof}

\section{Numerical experiments}\label{numsim}

We consider a portfolio selection and an inventory problem and compare the computational time required to solve
these problems using a simplex algorithm and DDP both with and without cut selection. In practice, the parameters of these problems (respectively the returns and the demands)
are not known in advance and stochastic optimization models are used for these applications. In our experiments where we intend to compare various solution methods for deterministic problems
of form \eqref{optconvexd}, we see these classes of problems as test problems for DDP with and without cut selection. Feasible instances of these problems can be
easily generated choosing randomly the problem parameters (initial wealth and the returns for the first; initial stock and demands for the second)
which are assumed to be known for our experiments.

\subsection{An inventory problem}

We consider the deterministic counterpart of the inventory problem described in Section 1.2.3 of \cite{shadenrbook}.
For each time step $t=1,\ldots, T$, of a given horizon of $T$ steps, on the basis of the inventory
level $y_t$ at the beginning of period $t$, we have to decide the quantity $x_t - y_t$ of a product
to buy so that the inventory level becomes $x_t$. Knowing the demand $\mathcal{D}_t$ for that product
for time $t$, the  inventory level is $y_{t+1}=x_t-\mathcal{D}_t$ at the beginning of period $t+1$.
The inventory level can become negative, in which case a backorder cost is paid.
We obtain the following optimization problem, of form \eqref{optconvexd}:
\begin{equation}\label{p2}
 \begin{array}{l}
  \min \;  \sum_{t=1}^{T} c_t (x_t - y_t ) + b_t (\mathcal{D}_t - x_t)_{+}  +  h_t (x_t - \mathcal{D}_t)_{+} \\
  x_t \geq y_t,\;t=1,\ldots,T,\\
  y_{t+1} = x_t -\mathcal{D}_t,\;t=1,\ldots,T-1,
\end{array}
\end{equation}
where $c_t$, $b_t$, $h_t$ are respectively  ordering, backorder penalty, and holding costs per unit,
at time $t$. We take $c_t =1.5+ \cos(\frac{\pi t}{6}), b_t=2.8>c_t$, $h_t=0.2$, $y_1=10$, and
$\mathcal{D}_t=5+\frac{1}{2}t$. For this problem, we can write the following DP equations
$$
\mathcal{Q}_t ( y_t ) =
\left\{
\begin{array}{l}
\min \; c_t (x_t - y_t ) +  b_t (\mathcal{D}_t - x_t)_{+}  +  h_t (x_t - \mathcal{D}_t)_{+} + \mathcal{Q}_{t+1} ( y_{t+1} )\\
y_{t+1} = x_t - D_t, x_t \geq y_t,
\end{array}
\right.
$$
for $t=1,\ldots,T$, with state vector $y_t$ of size one and $\mathcal{Q}_{T+1} \equiv 0$. 

We consider five algorithms to solve \eqref{p2}:
\begin{itemize}
\item  {\tt{Simplex}} as implemented by Mosek Optimization Toolbox \cite{mosek};
\item  {\tt{DDP}} as described in Section \ref{algorithmddp};
\item  {\tt{DDP}} with Level 1 cut selection ({\tt{DDP-CS-1}} for short);
\item  {\tt{DDP}} with the limited memory variant of Level 1 cut selection described in Section \ref{csmethods} ({\tt{DDP-CS-2}} for short).
\item  {\tt{Dynamic Programming}} ({\tt{DP}} for short) which computes in a backward pass for each $t=2,\ldots,T$, approximate values for $\mathcal{Q}_t$
at a discrete set of $N$ points equally spaced over the interval $[-100, 2000]$. These approximate values allow us to obtain accurate
representations of the recourse functions using affine interpolations between these points.
\end{itemize}
We stop algorithms {\tt{DDP}}, {\tt{DDP-CS-1}}, and {\tt{DDP-CS-2}} when the gap between the upper bound (given by the value of the objective at the feasible decisions) and the lower bound
(given by the optimal value of the approximate first stage problem , i.e., problem \eqref{forward1} written for $t=1$)
computed along the iterations is below a fixed tolerance $\varepsilon>0$. All subproblems to be solved along the iterations 
of {\tt{DP}}, {\tt{DDP}}, {\tt{DDP-CS-1}}, and {\tt{DDP-CS-2}} were solved using Mosek Optimization Toolbox \cite{mosek}.

We first take $T=600, N=2001$, and $\varepsilon=0.1$.
For this instance, the evolution of the upper and lower bounds computed along the iterations
of {\tt{DDP}} and {\tt{DDP-CS-2}} are represented in Figure \ref{figure0}.\footnote{This graph shows the typical evolution
of the upper (which tends to decrease) and lower (which is nondecreasing) bounds for DDP.}

\begin{figure}
\includegraphics[scale=0.8]{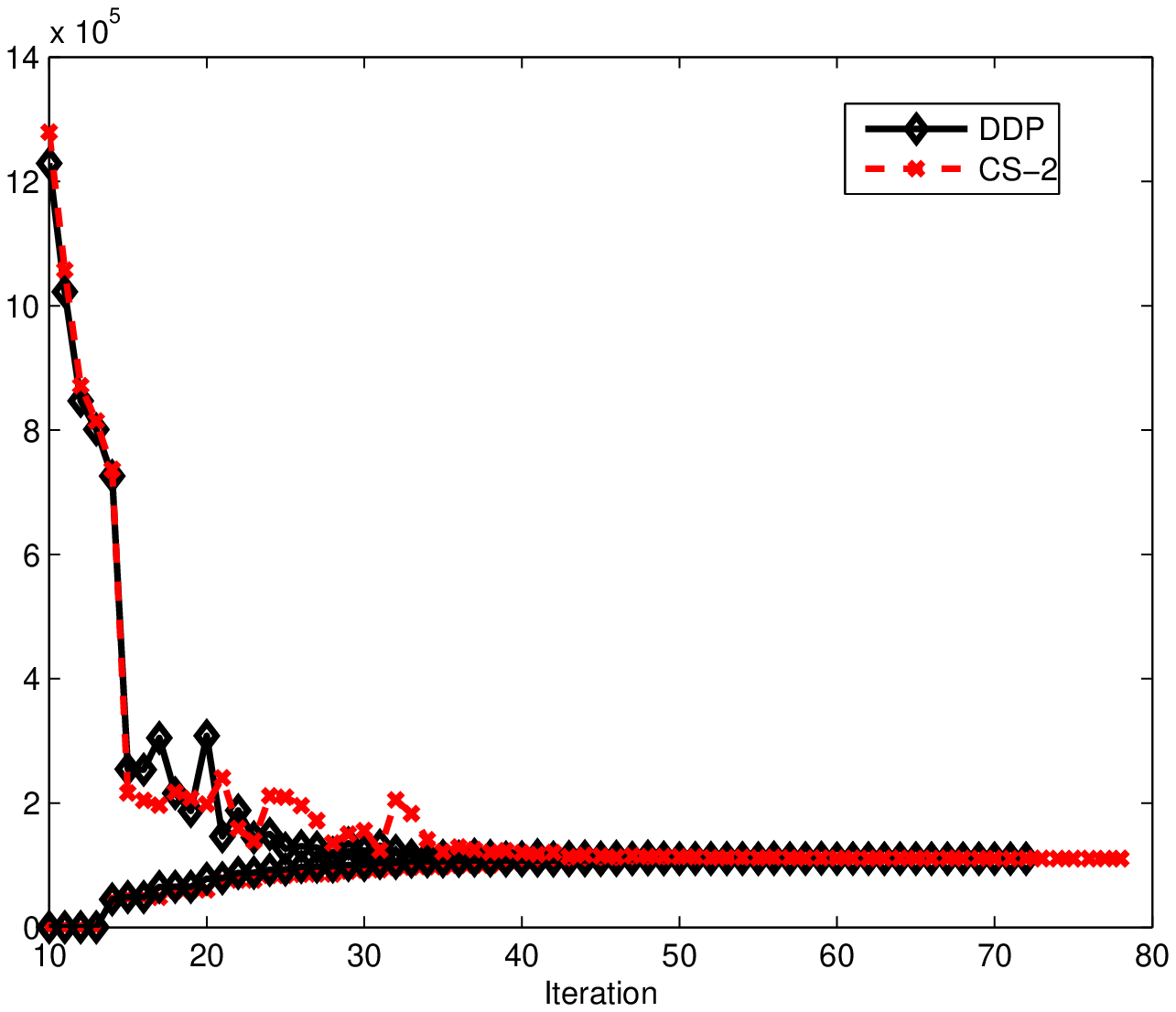}
\caption{Upper and lower bounds computed along the iterations of {\tt{DDP}} and {\tt{DDP-CS-2}} to solve an instance of \eqref{p2} for $T=600$ and  $\varepsilon=0.1$.}
\label{figure0}
\end{figure}

We report in Table \ref{table0}
the corresponding computational time and approximate optimal value obtained with each algorithm.
For {\tt{DDP}}, {\tt{DDP-CS-1}}, and {\tt{DDP-CS-2}}, we additionally report the number of iterations of the algorithm.
We see that on this instance, up to the precision $\varepsilon$, all algorithms provide the same approximate optimal value.
{\tt{Simplex}} is the quickest, {\tt{DP}} is by far the slowest,
and DDP and its variants are in between, with {\tt{DDP-CS-1}} and {\tt{DDP-CS-2}} variants requiring more iterations than DDP without cut selection
to converge. 
Though {\tt{DDP-CS-1}} and {\tt{DDP-CS-2}} run for the same number of iterations, {\tt{DDP-CS-2}} is quicker than {\tt{DDP-CS-1}} (and than {\tt{DDP}} too), due to the fact that
for many iterations and stages it selects much less cuts than {\tt{DDP-CS-1}}. More precisely, {\tt{DDP-CS-2}} is 13.7\% quicker than {\tt{DDP}}
and 20.4\% quicker than  {\tt{DDP-CS-1}}.
As an illustration, we plot on Figure \ref{figure2} 
approximations of value functions $\mathcal{Q}_t( y_t)$ for step $t=401$
estimated using DP and the cuts for these functions obtained at the last iteration of {\tt{DDP}}, {\tt{DDP-CS-1}}, and {\tt{DDP-CS-2}}.
Surprisingly, {\tt{DDP-CS-2}} only selected one cut at the last iteration.  
For {\tt{DDP-CS-1}}, we observe that
though 44 cuts are selected, they are all equal. This explains why graphically, we obtain on Figure \ref{figure2} 
the same cuts for {\tt{DDP-CS-1}} and {\tt{DDP-CS-2}} for $t=401$. Moreover, looking at Figure \ref{figure2}, we see that for $t=400$, 
the cuts built by {\tt{DDP}} have similar slopes and many of these cuts are useless, i.e., they are
below the approximate cost-to-go function on the interval $[-100, 2000]$. This explains why so many cuts are pruned.

Similar conclusions were drawn taking $T=96$ and several values of $t$, see Figure \ref{figure2}
for $T=96$ and $t=61$. 
For this experiment ($T=96$), at the last iteration, for $t=61$ {\tt{DDP-CS-2}} only selects one cut and {\tt{DDP}} 17 cuts (the number of iterations).
Surprisingly again, not only does {\tt{DDP-CS-1}} select more cuts (6 for $t=61$) but all these cuts are identical.

\begin{table}
\centering
\begin{tabular}{|c||c|c|c|}
\hline
{\tt{Algorithm}} &{\tt{Computational time}}& {\tt{Iterations}}& {\tt{Optimal value}}\\
\hline
{\tt{Simplex}} &   0.0868     &      -       & 110 660     \\
\hline
{\tt{DP}} &   20 623    &     -       &     110 660  \\
\hline
{\tt{DDP}}     &  76.0214  &       72     &   110 660\\
\hline
{\tt{DDP-CS-2}}     &  65.6318    &  78    &   110 660  \\
\hline
{\tt{DDP-CS-1}}    &  82.4705     &  78     &   110 660   \\
\hline
\end{tabular}
\caption{Computational time (in seconds) and approximate optimal values solving an instance of \eqref{p2}
with {\tt{Simplex}}, {\tt{DP}}, {\tt{DDP}}, {\tt{DDP-CS-1}}, and {\tt{DDP-CS-2}} with $T=600, N=2001$, and $\varepsilon=0.1$.}
\label{table0}
\end{table}

\begin{figure}
\begin{tabular}{ll}
\includegraphics[scale=0.45]{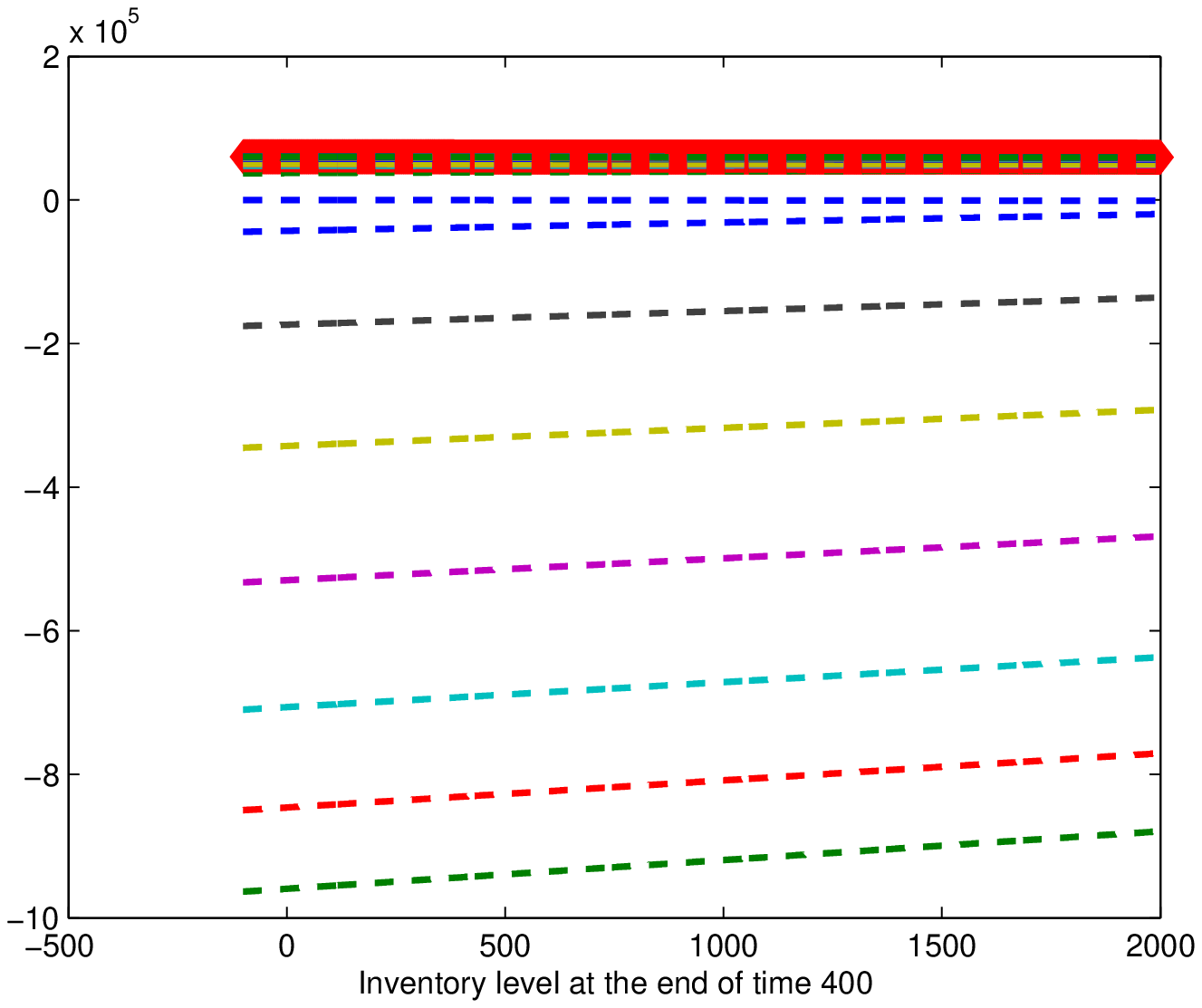}
&
\includegraphics[scale=0.45]{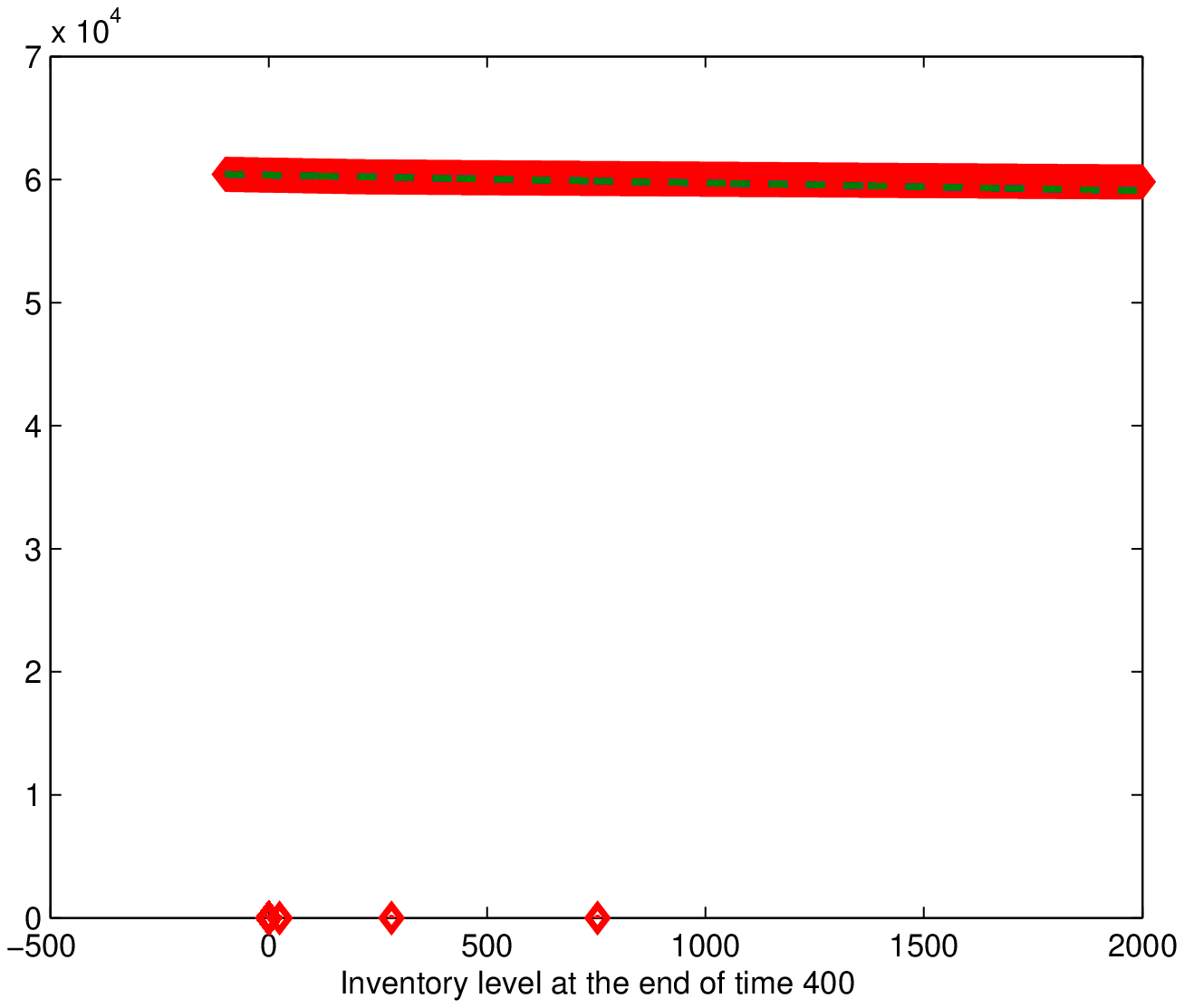}\\
\includegraphics[scale=0.45]{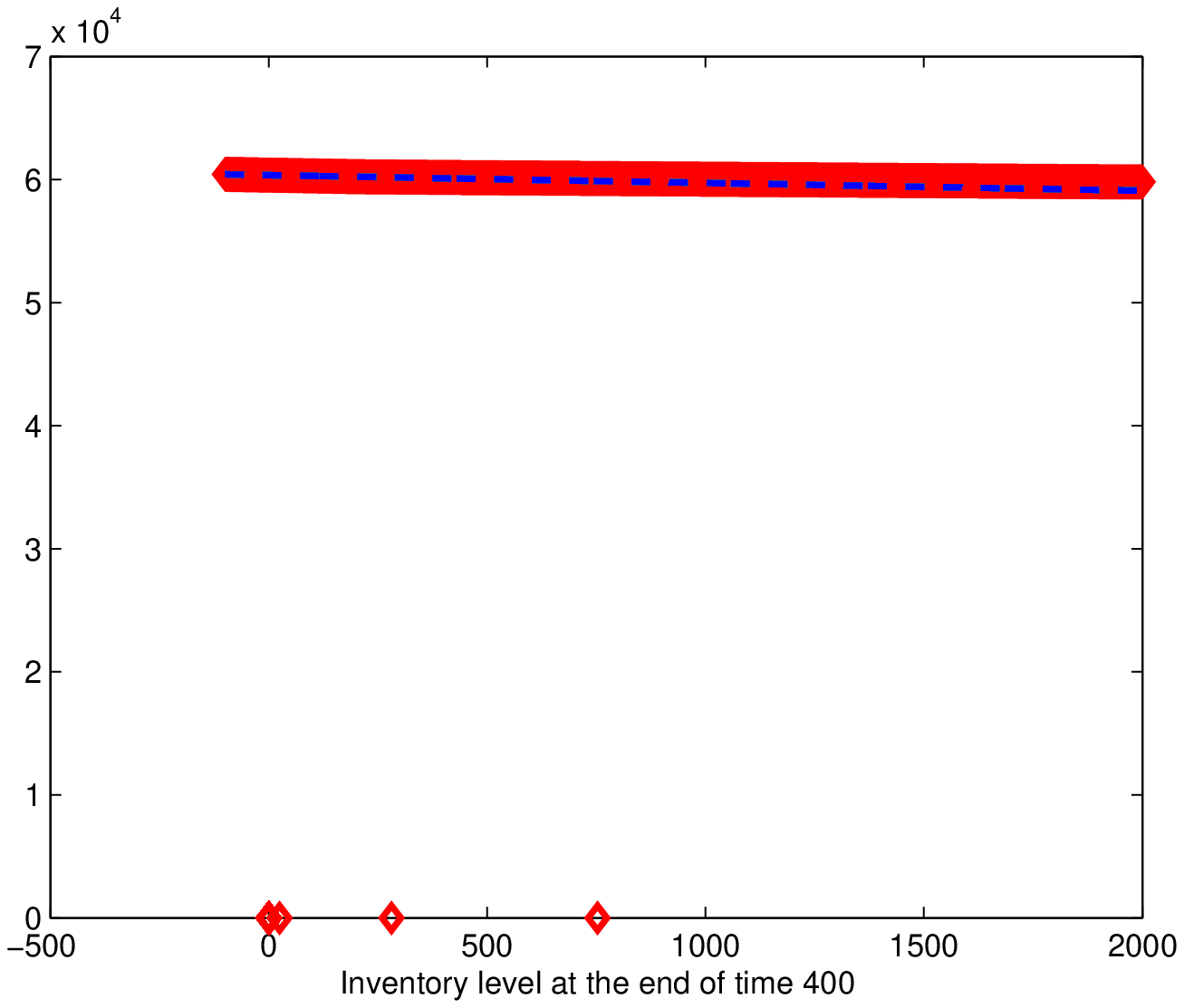}
&\includegraphics[scale=0.45]{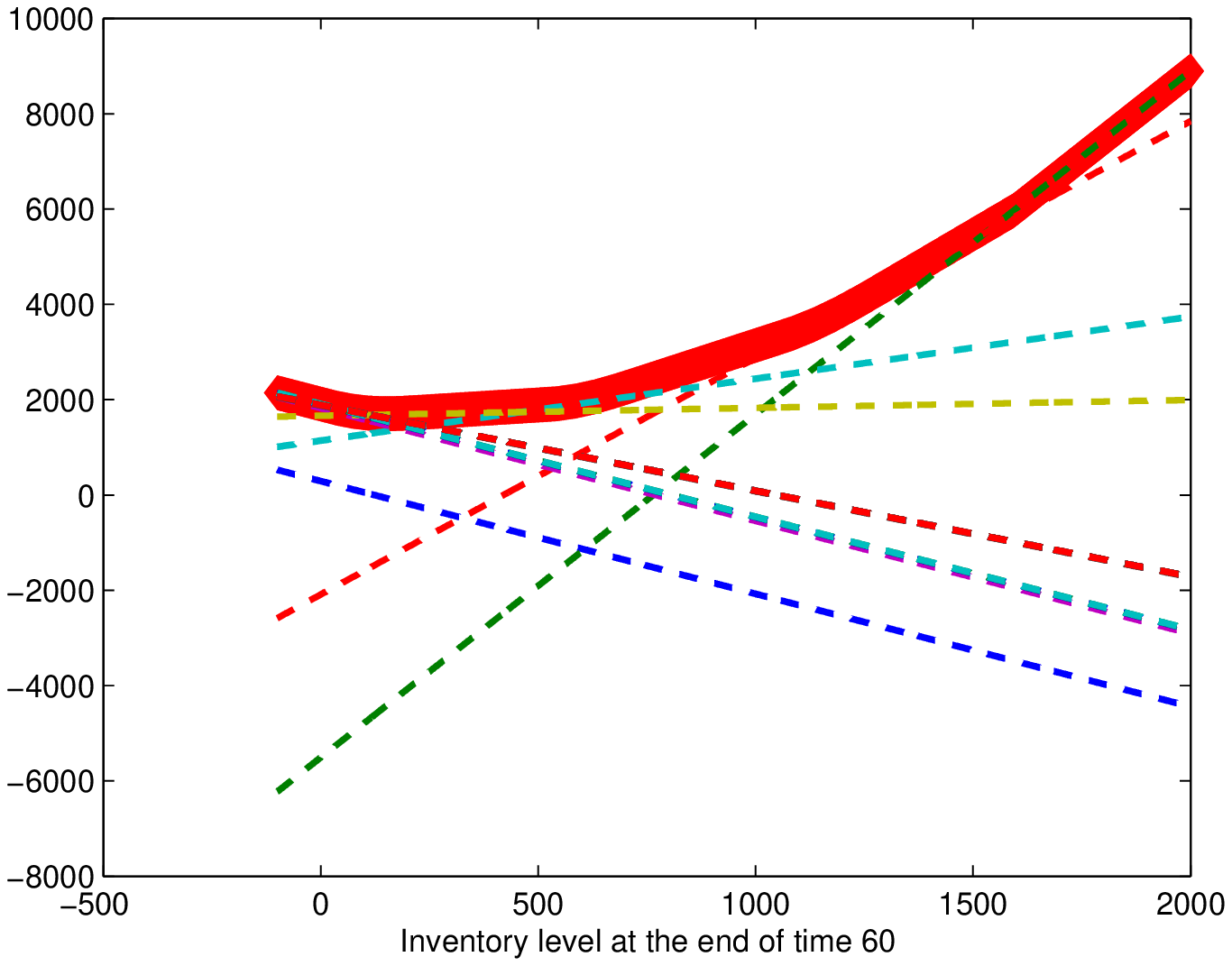}\\
\includegraphics[scale=0.45]{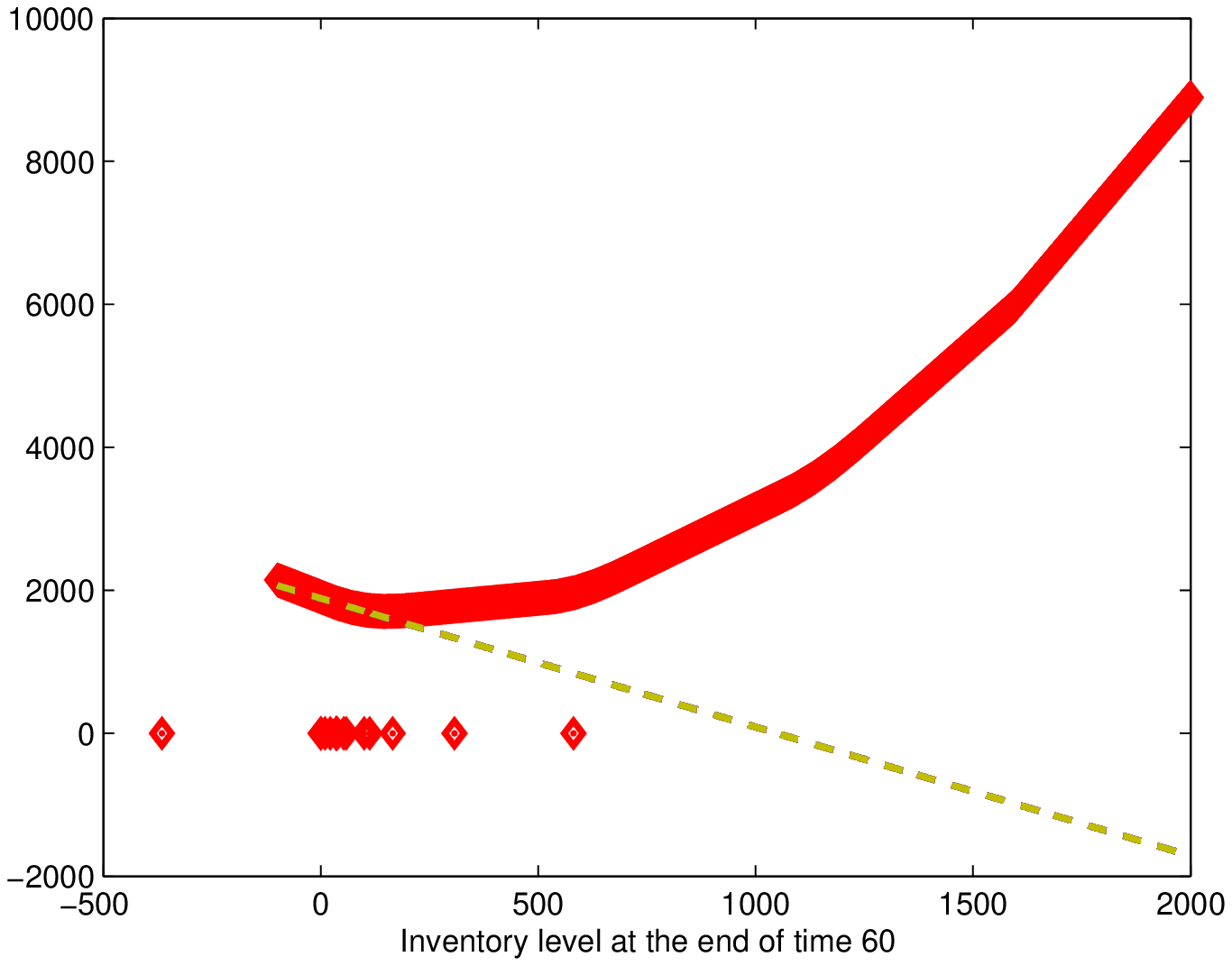}
&\includegraphics[scale=0.45]{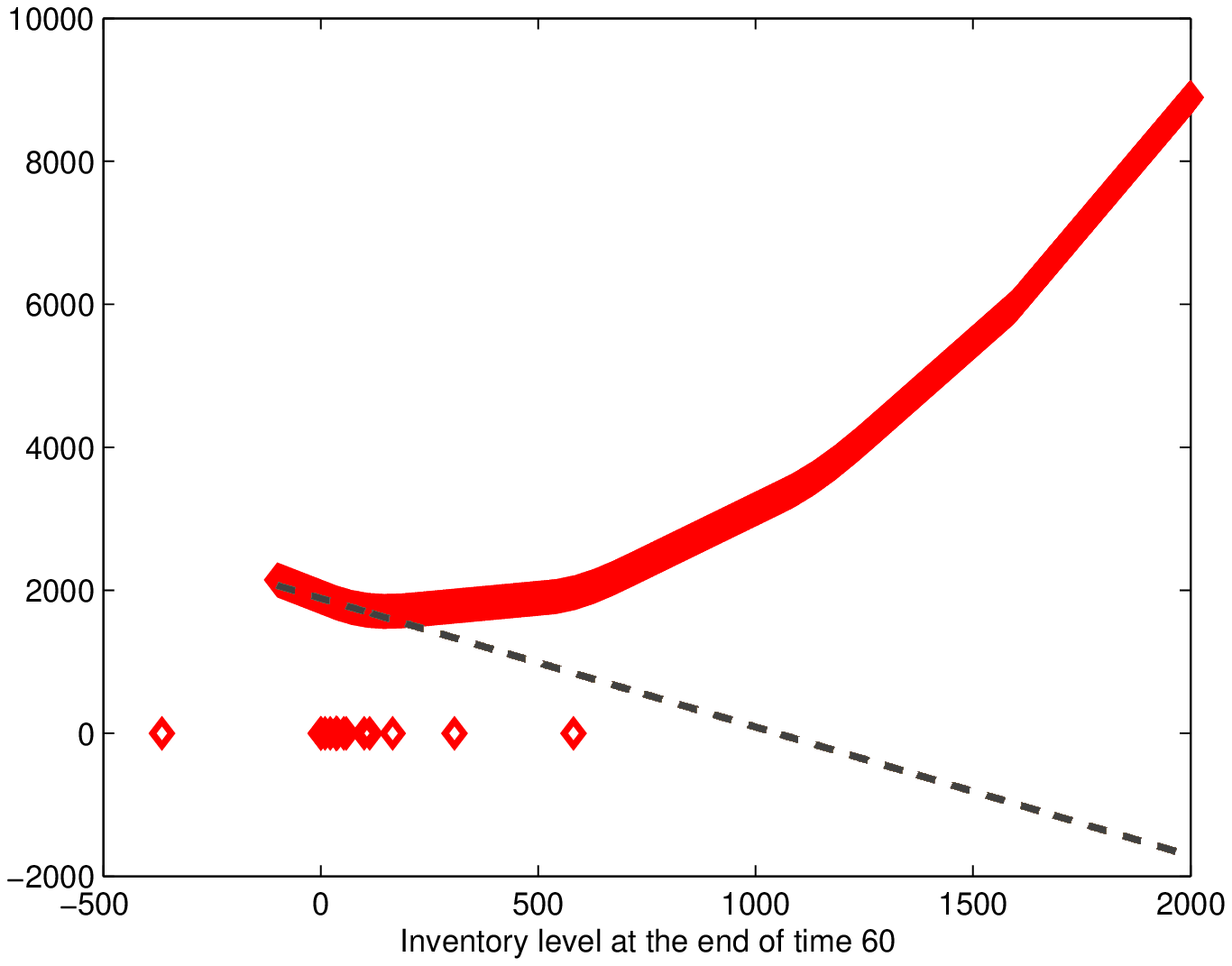}
\end{tabular}
\caption{Approximate value functions $\mathcal{Q}_t( y_t)$ estimated using DP (solid lines) and cuts for these functions (dashed lines) obtained at the last iteration
of {\tt{DDP}}, {\tt{DDP-CS-1}}, and {\tt{DDP-CS-2}}. Top left: $t=401, T=600$, {\tt{DDP}}.
Top right: $t=401, T=600$, {\tt{DDP-CS-1}}. Middle left: $t=401, T=600$, {\tt{DDP-CS-2}}. 
Middle right: $t=61, T=96$, {\tt{DDP}}. Bottom left: $t=61, T=96$, {\tt{DDP-CS-1}}. Bottom right: $t=61, T=96$, {\tt{DDP-CS-2}}.
Trial points less than $2000$ are represented by diamond markers on the $x$-axis.
}
\label{figure2}
\end{figure}

Finally, for DDP and its variants,
we report in Figure \ref{timeeps} the log-log evolution of the computational time for {\tt{DDP}} and its variants as a function of the number of steps $T$ (fixing $\varepsilon=0.1$)
running the algorithms for $T$ in the set $\{50,100,150,200,400,600,800, 1000, 5000\}$. We observe that for sufficiently large values of $T$,
{\tt{DDP-CS-2}} is the quickest.
On this figure, we also report the log-log evolution of the computational time as a function of $\varepsilon$
taking $\varepsilon$ in the set $\{0.00001,0.0001,0.001,0.01,0.1,1,5,10,50,100,500$, $1\,000,2\,000,3\,000,4\,000,5\,000,6\,000,7\,000, 8\,000$, \\$9\,000,10\,061\}$ and fixing 
$T=400$ for {\tt{DDP}} and {\tt{DDP-CS-2}} (the last value $10\,061$ of $\varepsilon$ taken corresponds to 25\% of the optimal value
of the problem). We see that unless $\varepsilon$ is very large, the computational time slowly decreases with $\varepsilon$. 

\begin{figure}
\begin{tabular}{ll}
\includegraphics[scale=0.5]{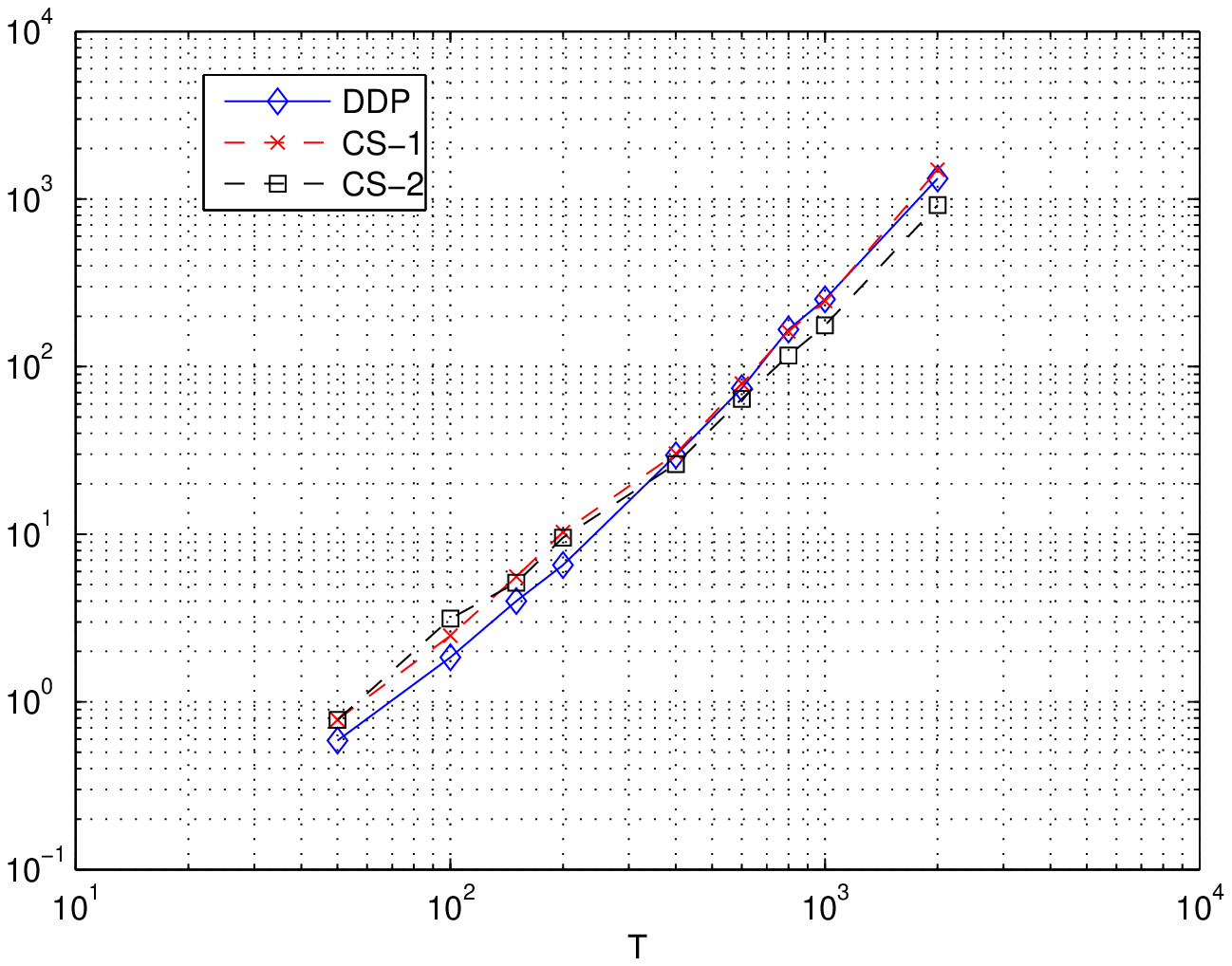}
&
\includegraphics[scale=0.5]{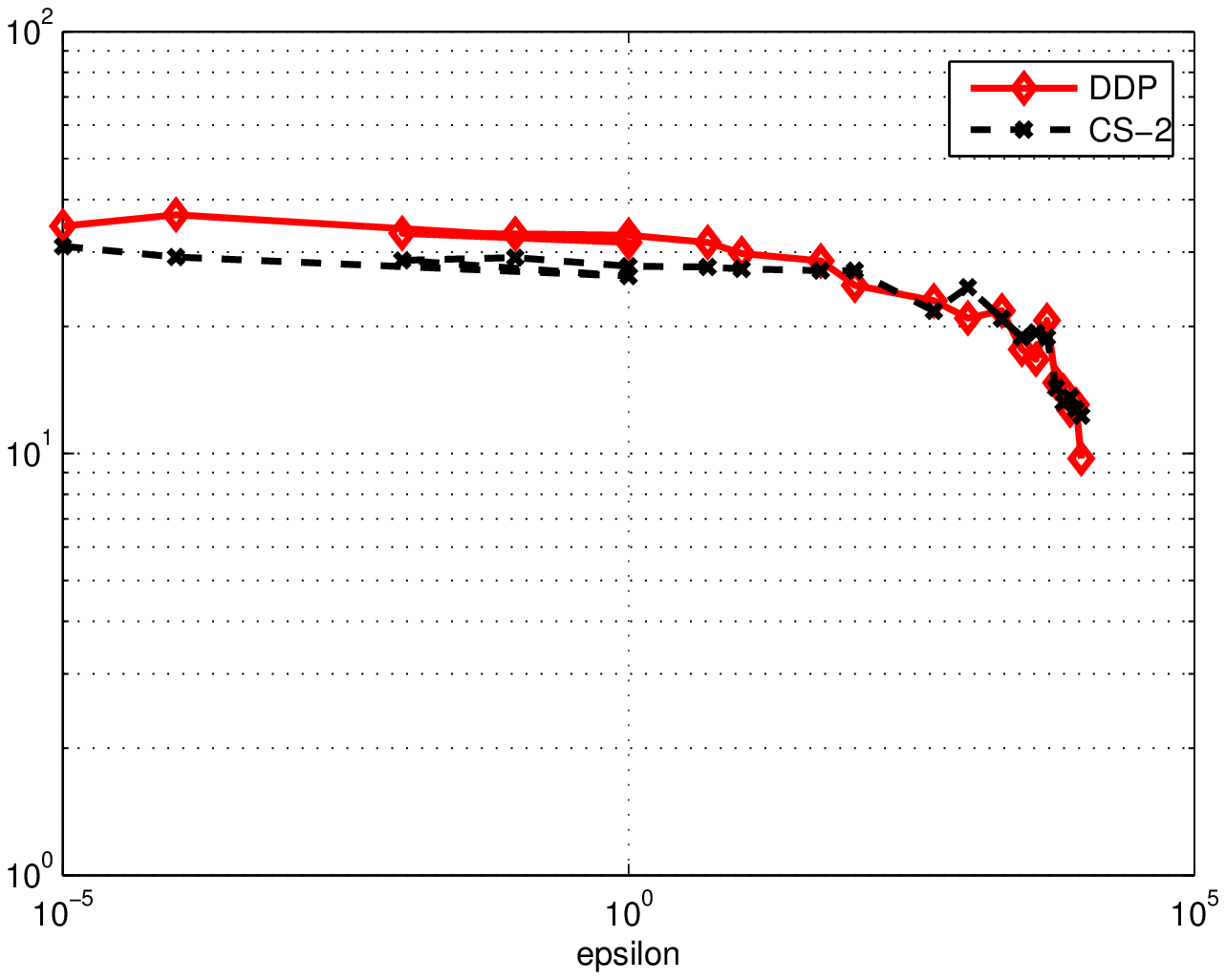}
\end{tabular}
\caption{Computational time (in seconds) as a function of $T$ and $\varepsilon$ for {\tt{DDP}} and its variants {\tt{DDP-CS-1}} and {\tt{DDP-CS-2}} ({\tt{CS-1}} and {\tt{CS-2}} for short)}.
\label{timeeps}
\end{figure}

\subsection{A portfolio problem}\label{portproblem}

Consider the portfolio problem with known returns
 \begin{equation}\label{p1}
   \begin{array}{l}
  \max \;  \sum\limits_{i=1}^{n+1}(1+r^i_{T}) x^i_{T} \\
  x^i_{t}=(1+r_{t-1}^i) x_{t-1}^i-y_t^i+z_t^i,\;      i=1,\ldots,n,\;t=1,\ldots,T,\\
  x_t^{n+1}=(1+r_{t-1}^{n+1})x_{t-1}^{n+1}
  +\sum\limits_{i=1}^{n}(1-\eta_i)y_t^i-  \sum\limits_{i=1}^{n}(1+\nu_i)z_t^i,\;t=1,\ldots,T,\\
  x_t^i \leq u_i \sum\limits_{i=1}^{n+1} (1+r_{t-1}^i) x_{t-1}^i,\;i=1,\ldots,n,\;t=1,\ldots,T,\\
  x_t^i\ge 0,\;i=1,\ldots,n+1,\;t=1,\ldots,T,\\
y_t^i\ge 0,z_t^i\ge 0,\;i=1,\ldots,n,\;t=1,\ldots,T.
 \end{array}
 \end{equation}
Here $n$ is the number of assets, asset $n+1$ is cash, $T$ is the number of time periods, $x_t^i$ is the value of asset $i=1,\ldots,n+1$ at the beginning of time 
period $t=1,\ldots,T$, $r_{t}^i$ is the return of asset $i$ for period $t$, $y_t^i$ is the amount of asset $i$ sold at time $t$, $z_t^i$ is the amount of asset $i$ bought  at 
time $t$, $\eta_i > 0$ and $\nu_i > 0$ are the respective transaction costs.
The components $x_0^i,i=1,\ldots,n+1$ are drawn independently of each other from the uniform distribution over the interval $[0,100]$
($x_0$ is the initial portfolio). The 
expression $\sum\limits_{i=1}^{n+1} (1+r_{t-1}^i) x_{t-1}^i$ is the budget available at the beginning of period $t$. 
The notation
$u_i$ in the third group of constraints is a parameter defining the maximal proportion that can be invested in financial security $i$.

We take $u_i=1$ for all $i$, and define several instances of \eqref{p1}
with $T=90$ and $n$ (the size of state vector $x_t$ in \eqref{optconvexd}) $\in$
$\{2,3,5,8,10,15,20,25,30,35,40$, $45,50,60,70,80$, $90,100,120,140,160,180$,\\$200,300,400,500, 600, 700, 800, 900, 1000, 1500\}$.
The returns of assets $i=1,\ldots,n$, are drawn independently of each other from the uniform distribution over the interval
$[0.00005, 0.0004]$, the return of asset $n+1$ is fixed to $0.0001$ (assuming daily returns, this amounts to a 3.72\% return over 365 days), and
transaction costs $\nu_i, \mu_i$ are set to $0.001$. 

In the notation of the previous section, these instances are solved using {\tt{Simplex}}, {\tt{DDP}}, {\tt{DDP-CS-1}},
and {\tt{DDP-CS-2}}. Since we consider in this section problems and state vectors of large size (state vector size up to $1500$ 
and problems with up to 405 090 variables and up to 675\,090 constraints [including box constraints]), {\tt{DP}} leads to prohibitive computational times and is not used.
The computational time required to solve these instances 
is reported in Table \ref{tabletime4} (all implementations are in Matlab, using 
Mosek Optimization Toolbox \cite{mosek} and for all algorithms the computational time includes the time to
fill the constraint matrices). For DDP and its variants, we take $\varepsilon=1$.
To check the implementation, we report in the Appendix, in Table \ref{tableopt4},
the approximate optimal values obtained for these instances with all algorithms.
We observe that
{\tt{Simplex}} tends to be quicker on small instances but when the state vector is large, DDP and its variants are much quicker.
The variants with and without cut selection have similar performances on most instances with {\tt{DDP-CS-1}} and {\tt{DDP-CS-2}}  tending to be quicker for the instances
with large $n$ (for instance for {\tt{DDP-CS-2}} compared to {\tt{DDP}}: 18 seconds instead of 21 seconds for $(T,n)=(90, 1500)$).  

If the DDP algorithms are very efficient, the variants with cut selection have not yielded a significant reduction in computational time
on these instances.  This can be partly explained by the small number of iterations (at most 5) needed for DDP algorithms to converge
on these runs. 

\begin{table}
\centering
\begin{tabular}{|c||c|c|c|c|c|c|c|c|}
\hline
$n$ &2&3&5&8&10&15&20&25\\
\hline
{\tt{Simplex}} &0.0847 &   0.0246    &    0.0413   &     0.0786   &     0.1070  &  0.2685 & 0.3831  &  0.6089\\
\hline
{\tt{DDP}}  &    0.2314  &  0.2440 &      0.4721  &     0.2845    &    0.5880  &  0.7401 & 0.6886 &   0.7738\\
\hline
{\tt{DDP-CS-2}} &   0.2652  &  0.2626   &    0.5021  &      0.3018    &    0.6129   & 0.6651  &0.7079  &  0.7800\\
\hline
{\tt{DDP-CS-1}}  &    0.2480  &  0.2630  &     0.5372   &     0.3105   &     0.6227  &  0.6888 & 0.7821 &   0.791\\
\hline
\hline
$n$ &30&35&40&45&50&60&70&80\\
\hline
{\tt{Simplex}}& 0.9670   & 1.3484 &   1.5092  &  1.9607  &  2.5572  &  3.5763 &   5.5887   & 7.2446\\
\hline
{\tt{DDP}} &          0.7735 &   0.7906 &   0.8337 &   0.8501 &   0.9177  &  0.9733  &  1.0739  &  1.1434\\
\hline
{\tt{DDP-CS-2}} &        0.8045  &  0.8198    &0.8676  &  0.8771 &   0.9520 &   1.0036  &  1.1068 &   1.1766\\
\hline
{\tt{DDP-CS-1}} &        0.8056  &  0.8267 &   0.8788   & 0.8816 &   0.9545  &  1.0094  &  1.1103  &  1.1795\\
\hline
\hline
$n$ &90&100&120&140&160&180&200&300\\
\hline
{\tt{Simplex}}&8.5491 &  12.1851 &  25.7598  & 36.5120  & 38.4824 &  46.0557 &  72.2004 & 129.2468  \\
\hline
{\tt{DDP}}& 1.1687  &  1.3347  &  1.4285   & 1.5798 &   1.7350 &   1.8178  &  2.0015   & 2.9424 \\   
\hline
{\tt{DDP-CS-2}} &   1.1833 &   1.3600 &   1.4464   & 1.5713  &  1.7479   & 1.8459  &  2.0315  &  2.9600\\    
\hline
{\tt{DDP-CS-1}}  &  1.1909  &  1.3571 &   1.4459  &  1.5728    &1.7509   & 1.8488  &  2.0413 &   2.9551  \\  
\hline
\hline
$n$ &400&500&600&700&800&900&1000&1500\\
\hline
{\tt{Simplex}}& 263.4561&  509.4387 &765.1   & 1084.8 & 1568  &  2080.9 &2922    &  10 449\\
\hline
{\tt{DDP}}  &    3.8153  &  5.7445  &  7.5   & 8.8 & 9.0    &   10.0 & 11    &    21 \\
\hline
{\tt{DDP-CS-2}} &   3.8197 &   5.7206& 7.4  &  8.6& 8.9   &    10.0  & 11 &  18\\
\hline
{\tt{DDP-CS-1}} &    3.8319 &   5.7955 & 7.4  &  8.7 & 9.0  &  10.1  &  11  &  18\\
\hline
\end{tabular}
\caption{Computational time (in seconds) to solve instances of \eqref{p1} with 
{\tt{Simplex}}, {\tt{DDP}}, {\tt{DDP-CS-1}}, and {\tt{DDP-CS-2}} for $T=90$ and several values of $n$.}
\label{tabletime4}
\end{table}

Finally, we apply {\tt{Simplex}}, {\tt{DDP}}, {\tt{DDP-CS-1}}, and {\tt{DDP-CS-2}} algorithms on a portfolio problem
of form \eqref{p1} using daily historical returns of $n$ assets of the S\&P 500 index for a period of $T$ days
starting on January 4, 2010, assuming these returns known (which provides us with the maximal possible return over this period
using these assets). We consider 8 combinations for the pair $(T,n)$:
$(25,50)$, $(50,50)$, $(75,50)$, $(100,50)$, $(25,100)$, $(50,100)$, $(75,100)$, and $(100,100)$.
We take again $\nu_i=\mu_i=0.001$ and $u_i=1$ for all $i$. The components $x_0^i,i=1,\ldots,n+1$ of $x_0$ 
are drawn independently of each other from the uniform distribution over the interval $[0,100]$.
The return of asset $n+1$ is $0.001$ for all steps.

For this experiment, we report the computational time and the approximate optimal values
in respectively Tables \ref{comptime5} and \ref{opt5}. Again the DDP variants are much quicker than simplex for problems
with large $T, n$.
The variants with cut selection are quicker, especially when $T$ and $n$ are large
(compared to {\tt{DDP}}, {\tt{DDP-CS-2}} is 5\% quicker for $(T,n)=(75, 100)$ and 6\% quicker for $(T,n)=(100, 100)$).
The number of iterations needed for the DDP variants to converge is given in Table \ref{iter5}.

\begin{table}
\centering
\begin{tabular}{|c||c|c|c|c|c|c|c|c|}
\hline
$(T,n)$ & $(25,50)$ & $(50,50)$  &  $(75,50)$ &  $(100,50)$   &  $(25,100)$   &  $(50,100)$   &  $(75,100)$  &  $(100,100)$    \\
\hline
{\tt{Simplex}} & 0.5642   & 1.2783  & 1.9757   &2.6020&2.6305&8.3588&9.2925& 12.1430\\
\hline
{\tt{DDP}}  & 0.3386  & 0.8705  & 1.6381  &3.2622&1.2789&2.7525&6.0626& 8.3887\\
\hline
{\tt{DDP-CS-2}}  &  0.3207 & 0.8138  & 1.2053  &2.8012&1.2520&2.2616&5.7634& 7.8885  \\
\hline
{\tt{DDP-CS-1}}  & 0.3206  & 0.8036  & 1.2022 &2.8261&1.2670&2.2801&5.8442& 9.1096\\
\hline
\end{tabular}
\caption{Computational time (in seconds) to solve instances of \eqref{p1} with 
{\tt{Simplex}}, {\tt{DDP}}, {\tt{DDP-CS-1}}, and {\tt{DDP-CS-2}} for several values of $(T,n)$ using historical returns of $n$ assets of S\&P 500 index.}
\label{comptime5}
\end{table}

\begin{table}
\centering
\begin{tabular}{|c||c|c|c|c|c|c|c|c|}
\hline
$(T,n)$ & $(25,50)$ & $(50,50)$  &  $(75,50)$ &  $(100,50)$   &  $(25,100)$   &  $(50,100)$   &  $(75,100)$  &  $(100,100)$    \\
\hline
{\tt{Simplex}} &  4 754.3  & 10 600   & 24 525&46 496&13 584&53 576&178 190& 403 800\\
\hline
{\tt{DDP}}  & 4 754.2  & 10 600  &24 525&46 496&13 584&53 576&178 190& 403 800\\
\hline
{\tt{DDP-CS-2}}  &  4 754.2 &  10 600 &24 525&46 496&13 584&53 576&178 190& 403 800\\
\hline
{\tt{DDP-CS-1}}  & 4 754.2  & 10 600  &24 525&46 496&13 584&53 576&178 190& 403 800\\
\hline
\end{tabular}
\caption{Approximate optimal value of instances of \eqref{p1} obtained with 
{\tt{Simplex}}, {\tt{DDP}}, {\tt{DDP-CS-1}}, and {\tt{DDP-CS-2}} for several values of $(T,n)$ using historical returns of $n$ assets of S\&P 500 index.}
\label{opt5}
\end{table}

\begin{table}
\centering
\begin{tabular}{|c||c|c|c|c|c|c|c|c|}
\hline
$(T,n)$ & $(25,50)$ & $(50,50)$  &  $(75,50)$ &  $(100,50)$   &  $(25,100)$   &  $(50,100)$   &  $(75,100)$  &  $(100,100)$    \\
\hline
{\tt{DDP}}  &  6 &   7&7&11&12&13& 17&17\\
\hline
{\tt{DDP-CS-2}}  & 5  & 6  &6&10&13&12&20& 23\\
\hline
{\tt{DDP-CS-1}}  &  5 &  6 &6&10&13&12&20& 20\\
\hline
\end{tabular}
\caption{Number of iterations of DDP method with 
{\tt{DDP}}, {\tt{DDP-CS-1}}, and {\tt{DDP-CS-2}} for  instances of \eqref{p1} built using historical returns of $n$ assets of S\&P 500 index.}
\label{iter5}
\end{table}

\section{Conclusion} 

We proved the convergence of DDP with cut selection for cut selection strategies satisfying Assumption (H2).
This assumption is satisfied by the {\em{Territory algorithm}} \cite{pfeifferetalcuts}, the Level $H$ cut selection \cite{dpcuts0},
as well as the limited memory variant we proposed.

Numerical simulations on an inventory and a portfolio problem have shown that DDP can be much quicker than simplex
and that cut selection can speed up the convergence of DDP for some large scale instances.
For instance, for the simulated portfolio problem with $(T,n)=(90, 1500)$ of Section \ref{numsim}, DDP was about 500 times
quicker than simplex. We also recall that for the inventory problem solved for $T=600$, {\tt{DDP-CS-2}} (that uses the limited memory variant) was 13.7\% quicker than {\tt{DDP}}
and 20.4\% quicker than  {\tt{DDP-CS-1}}. A possible explanation is that {\tt{DDP}} produced 
a large number of dominated/useless cuts, which are below the useful cuts on the whole range of admissible
states. We expect more generally the cut selection strategies to be helpful when the number of iterations of DDP
(and thus the number of cuts built) is large.

For the porfolio management instances with large $n$,
{\tt{DDP-CS-1}} and {\tt{DDP-CS-2}}  tend to be quicker (for the tests using historical daily returns of S\&P 500 index,
compared to {\tt{DDP}}, {\tt{DDP-CS-2}} is 5\% quicker for $(T,n)=(75, 100)$ and 6\% quicker for $(T,n)=(100, 100)$).

The proof of this paper can be extended in several ways:
\begin{itemize}
\item As in \cite{guigues2014cvsddp}, we can consider the case where $f_t$ and the constraint set for step $t$ depend 
not only on $x_{t-1}$ but on the full history
of decisions $x_1, x_2, \ldots, x_{t-1}$. In this case, $\mathcal{Q}_{t}$ also depends on
the full history
of decisions $x_1, x_2, \ldots, x_{t-1}$. This larger problem class was considered for risk-averse convex programs without cut selection
in \cite{guigues2014cvsddp}. 
\item It is also possible to consider a variant of DDP that uses a backward pass to compute
the cuts, i.e., that computes the cuts using at iteration $k$ function $\mathcal{Q}_{t+1}^k$ instead of
$\mathcal{Q}_{t+1}^{k-1}$  in \eqref{forward1}. The corresponding convergence proof can be easily obtained following the
convergence proof of Theorem \ref{convddpth}. 
\item Extend the convergence proof for decomposition methods with cut selection
to solve multistage risk-averse nonlinear stochastic optimization problems.
\end{itemize}

\section*{Acknowledgments} The author's research was 
partially supported by an FGV grant, CNPq grant 307287/2013-0, 
FAPERJ grants E-26/110.313/2014 and E-26/201.599/2014.
We would also like to thank the reviewers for beneficial comments and suggestions.

\addcontentsline{toc}{section}{References}
\bibliographystyle{plain}
\bibliography{Risk_Averse_SDDP}

\section*{Appendix}

\begin{table}[H]
\centering
\begin{tabular}{|c||c|c|c|c|c|c|c|c|}
\hline
$n$ &2&3&5&8&10&15&20&25\\
\hline
{\tt{Simplex}}& 63.9523 & 265.7561 & 300.7989  &  527.9508 &  515.5728 & 649.9466 &  1 156 &  1 138.4\\ 
\hline
{\tt{DDP}}  &    63.9314 & 265.6687 & 300.7989 & 527.9489 & 515.5728 &    649.8515 & 1 156 &  1 138.4\\
\hline
{\tt{DDP-CS-2}}   & 63.9314 & 265.6687 & 300.7989 & 527.9489 & 515.5728 &    649.8515 &  1 156 &  1 138.4\\
\hline
{\tt{DDP-CS-1}}  & 63.9314 & 265.6687 & 300.7989 & 527.9489 & 515.5728 &    649.8515 & 1 156  & 1 138.4\\
\hline
\hline
$n$ &30&35&40&45&50&60&70&80\\
\hline
{\tt{Simplex}}& 1 620.3 &   1 877.8  &  2 127.4  &  2 037.0  & 2 183.4 &  2 894.3 &  3 420.6   & 4 101.5\\
\hline
{\tt{DDP}} &  1 620.3  &  1 877.8  &  2 127.4  &  2 037.0 &  2 183.4  & 2 894.3  & 3 420.6  &  4 101.5\\      
\hline
{\tt{DDP-CS-2}} &  1 620.3  &  1 877.8  &  2 127.4  &  2 037.0 &  2 183.4 &  2 894.3 &  3 420.6 &   4 101.5\\  
\hline
{\tt{DDP-CS-1}} &  1 620.3   & 1 877.8  &  2 127.4  &  2 037.0  & 2 183.4 &  2 894.3 &  3 420.6 &   4 101.5\\
\hline
\hline
$n$ &90&100&120&140&160&180&200&300\\
\hline
{\tt{Simplex}}& 4 919  &  5 147  &  6 306 &   6 835 &   8 254 &   9 151 &  9 773 &  16 016 \\   
\hline
{\tt{DDP}}& 4 919  &  5 147  &  6 306 &   6 835 &   8 254 &   9 151 &  9 773 &  16 016 \\   
\hline
{\tt{DDP-CS-2}} &  4 919  &  5 147  &  6 306 &   6 835 &   8 254 &   9 151 &  9 773 &  16 016 \\   
\hline
{\tt{DDP-CS-1}}  & 4 919  &  5 147  &  6 306 &   6 835 &   8 254 &   9 151 &  9 773 &  16 016 \\   
\hline
\hline
$n$ &400&500&600&700&800&900&1000&1500\\
\hline
{\tt{Simplex}}& 21 050  &  25 863 &30 410 &   36 229 &   41 145  &  45 024  &  50 994 &   76 695\\
\hline
{\tt{DDP}}  &    21 050  &  25 863 &30 410 &   36 229 &   41 145  &  45 024  &  50 994 &   76 695\\
\hline
{\tt{DDP-CS-2}} &  21 050  &  25 863 &30 410 &   36 229 &   41 145  &  45 024  &  50 994 &   76 695\\
\hline
{\tt{DDP-CS-1}} &  21 050  &  25 863 &30 410 &   36 229 &   41 145  &  45 024  &  50 994 &   76 695\\
\hline
\end{tabular}
\caption{Optimal values of instances of \eqref{p1} considered in Section \ref{portproblem} solved with
{\tt{Simplex}}, {\tt{DDP}}, {\tt{DDP-CS-1}}, and {\tt{DDP-CS-2}} for $T=90$ and several values of $n$.}
\label{tableopt4}
\end{table}
  
\end{document}